\documentclass[11pt]{amsart}
\usepackage{mathrsfs}
\usepackage[mathcal]{euscript}
\usepackage{graphicx}
\usepackage{amsthm,amscd}
\usepackage{calc}
\usepackage{color}
\usepackage{mathrsfs,dsfont}
\usepackage{fancyhdr,amscd}
\usepackage{anysize}
\usepackage{amsmath,amssymb,amsfonts}
\usepackage{epsfig,enumerate}
\usepackage{latexsym}
\usepackage{indentfirst, latexsym}
\usepackage{graphics}
\usepackage[all,poly,knot]{xy}
\usepackage[pagebackref]{hyperref}
\providecommand{\U}[1]{\protect\rule{.1in}{.1in}}

\numberwithin{equation}{section}

\def\blambda{\Lambda}

\def\bomega{\Omega}
\def\vep{\varepsilon}
\def\dbar{\bar\partial}
\theoremstyle{plain}
\newtheorem{thm}{Theorem}[section]%ÉùÃ÷¶¨Àí»·¾³£¬´Ë´¦¡°¶¨Àí¡±°üÀ¨¶¨Àí¶¨ÒåÒýÀí¹«ÀíÃüÌâµÈ¡£
\newtheorem{lemma}[thm]{Lemma}

\newtheorem{prop}[thm]{Proposition}
\newtheorem{cor}[thm]{Corollary}

\theoremstyle{definition}
\newtheorem{defn}[thm]{Definition}
\newtheorem{ex}[thm]{Example}
\newtheorem{rem}[thm]{Remark}
\newtheorem{notation}[thm]{Notation}
\newcommand{\comment}[1]{}

\usepackage{fancyhdr}
\begin{document}
%µ÷Óà ÏÔʾÐÐÊýµÄºê°ü
%\linenumbers
\title{$L^2$ extension of $\bar\partial$-closed forms on weakly pseudoconvex K\"ahler manifolds}

\author[Jian Chen]{Jian Chen}
\author[Sheng Rao]{Sheng Rao}

\address{Jian Chen, School of Mathematics and Statistics, Wuhan  University,
Wuhan 430072, People's Republic of China}
\email{jian-chen@whu.edu.cn}

\address{Sheng Rao, School of Mathematics and Statistics, Wuhan  University,
Wuhan 430072, People's Republic of China;
Universit\'{e} de Grenoble-Alpes, Institut Fourier (Math\'{e}matiques)
UMR 5582 du C.N.R.S., 100 rue des Maths, 38610 Gi\`{e}res, France}
\email{likeanyone@whu.edu.cn, sheng.rao@univ-grenoble-alpes.fr}
\dedicatory{One world, one fight}
\thanks{Both authors are partially supported by NSFC (Grant No. 11671305, 11771339, 11922115) and the Fundamental Research Funds for the Central Universities  (Grant No. 2042020kf1065).
}

\date{\today}
%------------Ö÷Ìâ·ÖÀàÐè²éѯ------------\textcolor{red}{need to modify}

%32C10 (1973-1999) Complex manifolds
%32C17 (1991-1999) K?hler geometry
%32C25 (1973-now) Analytic subsets and submanifolds
%32C30 (1973-now) Integration on analytic sets and spaces, currents [For local theory, see 32A25 or 32A27]
%32C35 (1973-now) Analytic sheaves and cohomology groups [See also 14Fxx, 18F20, 55N30]
%32C36 (1980-now) Local cohomology of analytic spaces
%32D15 (1973-now) Continuation of analytic objects
%32E10 (1973-now) Stein spaces, Stein manifolds
%32F05 (1973-1999) Plurisubharmonic functions and generalizations [See also 31C10]
%32J27 (1991-now) Compact K?hler manifolds: generalizations, classification
%32Q15 (2000-now) K?hler manifolds
%\subjclass[2010]{Primary 32L10; Secondary 32Q15}

\begin{abstract}
Combining V. Koziarz's observation about the regularity of some modified  section related to the initial  extension with J. McNeal--D. Varolin's regularity argument, we generalize two theorems of McNeal--Varolin for the $L^2$ extension of $\bar\partial$-closed high-degree forms on a Stein manifold to the weakly pseudoconvex K\"ahler case under mixed positivity conditions.
\end{abstract}
\maketitle

\setcounter{tocdepth}{1}
\tableofcontents

\section{Introduction: main results and applications}\label{intro}
As is well known, the task of the classical Ohsawa--Takegoshi theorem \cite{[OT]} is to
extend a holomorphic object from some lower or same dimensional analytic subvariety to
the ambient space with some $L^2$ estimate involved.   In recent years, problems of extending
holomorphic sections have been treated almost completely in the category of analytic manifolds. So it is natural to ask whether these extensions are feasible for $\bar\partial$-closed forms of high degree, which is a  natural broadening of the classical Ohsawa--Takegoshi--Manivel extension theorem for holomorphic sections of line bundles.

Many interesting works appear along this line, such as \cite{Mn93,Dm00,Koz11,ZGZ12,B12,BPZ15,M-V19,ZZ19}, etc.
The biggest difficulty of this problem is the regularity issue for  solutions of  related $\bar\partial$-equation for  high degree  because $\bar\partial$~operator for high degree is no longer hypoelliptic. For solving the regularity issue, two main methods were adopted: minimizer method and Leray's isomorphism method.

L. Manivel \cite{Mn93} firstly considered this problem, while his proof has difficulty to complete due to the use of a singular weight and the failure of regularity for the solution of the related $\bar\partial$-equation. Then J.-P.  Demailly \cite{Dm00} suggested an approach
to overcome this difficulty, while no one seems to have implemented his  program yet completely. V. Koziarz \cite{Koz11} used the Leray's isomorphism to reduce the extension of high-degree forms to the classical  zero-degree  case and thereby deduced extensions of cohomology classes. B. Berndtsson \cite{B12} applied the minimizer method by solving a $\bar\partial$-equation for a current to get the related extension theorem on compact manifolds. In \cite{M-V19}, J. McNeal--D. Varolin made use of the Kohn solution, to handle the well-known regularity issues on a Stein or essentially Stein manifold (i.e., a K\"ahler manifold that
 becomes Stein after a hypersurface is removed from it). Furthermore, L. Zhu--Q. Guan--X. Zhou \cite{ZGZ12} and L. Baracco--S. Pinton--G. Zampieri \cite{BPZ15} also got some results in some special cases.

It is natural to ask whether we can establish extension theorems  of $\dbar$-closed forms of high degree on general weakly pseudoconvex  K\"ahler manifolds. Recall that a complex manifold is \emph{weakly pseudoconvex} if it admits a smooth plurisubharmonic exhaustion function. As is well known, there is a richer cohomology theory on   weakly pseudoconvex  K\"ahler manifolds than that on Stein manifolds. For example, it is interesting to find
 out the  conditions for a K\"ahler family (the total space of which is a weakly pseudoconvex  K\"ahler manifold, more precisely, a holomorphic convex manifold) to admit the deformation invariance of  higher cohomology of the pluricanonical bundle.  In fact, this question is our initial motivation to study the problem of $L^2$ extension of $\bar\partial$-closed forms on weakly pseudoconvex K\"ahler manifolds.  By the way,  even for a (fiberwise) projective family,   higher cohomology of the pluricanonical bundle may not be  deformation invariant \cite{HL08}.

Combining Koziarz's observation \cite{Koz11} about the regularity of some modified  section related to the initial ambient extension with McNeal--Varolin's regularity argument \cite{M-V19} for the extension of high-degree forms, we generalize the $L^2$ extension theorems (cf. Theorems \ref{mv Ambient} and \ref{mv Intrinsic}) of McNeal--Varolin \cite{M-V19} on a Stein manifold to a weakly pseudoconvex K\"ahler manifold under mixed positivity conditions.

\begin{thm}[Ambient $L^2$ extension]\label{ambient extension}
 On a weakly pseudoconvex $n$-dimensional K\"ahler manifold $(X,\omega)$, let the smooth hypersurface $Y\stackrel{\iota}{\hookrightarrow} X$ be the zero set of a holomorphic section $s\in H^0(X,E)$ of a smooth Hermitian holomorphic line bundle $(E,e^{-\lambda})$ and $(L,e^{-\varphi})$ a smooth Hermitian holomorphic line bundle. Assume that for any $0\leq q \leq n-1$, the inequalities hold on $X$
\begin{eqnarray}
&&\sqrt{-1}\partial\bar\partial(\varphi-\lambda)\wedge\omega^q\geq \sigma\omega\wedge\omega^q \label{assone},\\
&&\sqrt{-1}\partial\bar\partial(\varphi-(1+\delta)\lambda)\wedge\omega^q\geq 0 ,\label{asstwo}\\
&&|s|^2e^{-\lambda}\leq 1, \label{assthree}
\end{eqnarray}
where $\sigma$ is some positive \footnote{In  \cite[Theorem 1.1]{M-V19}, the $\sigma$ is taken as zero. So the theorem here cannot cover \cite[Theorem 1.1]{M-V19} completely. The error term method of Demailly \cite{Dm00} of solving $\dbar$-equations is always used to overcome the lack of sufficient positivity for the extension problems of holomorphic objects. It seems interesting to  improve the condition \eqref{assone} to be the semi-positive one like \cite[Theorem 1.1]{M-V19}.}
%(However, this method can not work when dealing with extension problems of $\dbar$-closed forms mainly due to the lack of ellipticity of $\dbar$ operator.)
lower semi-continuous function  and $0<\delta\leq 1$ is some constant.
Then there is a universal constant $C > 0$ such that for any smooth
section $f$ of the bundle  $(K_X\otimes L\otimes \Lambda^{0,q}T^\star_X)|_Y\rightarrow Y$,  satisfying
\[
\bar\partial(\iota^*f)=0 \quad and \quad \int_Y\frac{|f|^2_{\omega} e^{-\varphi}}{|ds|^2_{\omega}e^{-\lambda}}dV_{Y,\omega}<\infty,
\]
there exists a smooth $\bar\partial$-closed $K_X\otimes L$-valued $(0, q)$-form $F$ on $X$  with
$$F|_Y=f \quad and \quad \int_X|F|^2_\omega e^{-\varphi}dV_{X,\omega} \leq \frac{C}{\delta}\int_Y\frac{|f|^2_{\omega} e^{-\varphi}}{|ds|^2_{\omega}e^{-\lambda}}dV_{Y,\omega}<\infty. $$
Here we denote by $\Lambda^{r,s}T^\star_X$ the bundle of differential forms of bidegree $(r,s)$ on $X$ and similarly for others.
\end{thm}

Note that since $(K_X\otimes L\otimes \Lambda^{0,q}T^\star_X)|_Y$ does not admit a natural notion of $\bar\partial$, the above $\iota^*f$  is not the usual pullback of differential forms, but induced as Definition \ref{pullback}. And the \emph{positivity conditions} \eqref{assone} \eqref{asstwo} hold \emph{in the sense of $(q+1,q+1)$-forms} as follows:  a real $(1,1)$-form $\theta$ on $X$ satisfies the positivities on $X$
\begin{equation}\label{pos}
\theta\wedge\omega^q> (\text{resp. $\geq$})\ 0
\end{equation}
if and only if $$\lambda_{1}+\cdots+\lambda_{q+1}> (\text{resp. $\geq$})\ 0,$$
where $\lambda_{1} \leqslant \cdots \leqslant \lambda_{n}$ are the eigenvalues of $\theta$ with respect to $\omega$ at any point $x\in X$. Moreover, they are both equivalent to  that
$\langle[\theta,\blambda_{\omega}] \beta, \beta\rangle_{\omega}$ is positive (resp. semipositive) at any point $x\in X$ for any  $(n, q+1)$-form $\beta(x)\neq 0$ (see Lemmata \ref{eigen} and \ref{less than q eigen} for a better understanding).

%For the case of semi-positivity conditions, the error term method of Demailly \cite{Dm00} of solving $\dbar$-equations is always used to overcome the lack of sufficient positivity. Unfortunately, we find no good methods to combine   the regularity argument of McNeal--Varolin  with the error term method to get a result under the semi-positivity conditions. The strict positivity condition \eqref{assone} is mainly applied to force the Dirichlet semi-norm to be a norm on some relatively compact domains of the weakly pseudoconvex  K\"ahler manifold and to control well  the estimate \eqref{second part}.
%However, since there are a variety of vanishing theorems, Theorem \ref{ambient extension} may not tell us anything valuable if one considers the extension problem of cohomology as $L$ is very highly positively curved.

Theorem \ref{ambient extension} is clearly different from  \cite[Theorem 3.1]{B12} which is in an intrinsic sense (see Definition \ref{Two} for this notion) by the adjunction formula.   We give an example satisfying the conditions of Theorem \ref{ambient extension}, but not in the setting of  \cite[Theorem 1]{M-V19} or the essentially Stein case of \cite{M-V19} (cf. \cite[p. 425]{M-V19}).

\begin{ex}\label{ex1}
Let $\left(\mathbb{B}^{m}, \omega_{1}\right)$ be the unit ball in $\mathbb{C}^{m}$ equipped with the Euclidean metric $\omega_{1}=\sqrt{-1}\partial\dbar(z_1^2+\cdots+z_m^2)$, and $\left(Y, \omega_{2}\right)$ a $k$-dimensional compact K\"ahler manifold which does  not have any closed complex hypersurfaces.   By the heredity property of weakly  pseudoconvexity,  $X:=\mathbb{B}^{m} \times Y$ is a weakly pseudoconvex K\"ahler manifold equipped with the natural K\"ahler metric $\omega:=\pi_1^*\omega_{1}+\pi_2^*\omega_{2}$. As $X$ admits a compact submanifold $Y$ which contains no hypersurfaces, it is neither a Stein nor an  essentially Stein manifold. Apparently, $X$ is not a compact manifold, either.
Let
$$L=E=\mathcal{O}_X, \quad s=z_1, \quad \varphi=|z|^2, \quad \sigma=\frac{q+1-k}{2(q+1)}, \quad and \quad \lambda\equiv 0,$$
where $z_1$ is the first global coordinate function on $\mathbb{B}^{m}$, $|z|^2:=z_1^2+\cdots+z_m^2$.
Then the above setting satisfies \eqref{assone}, \eqref{asstwo} and \eqref{assthree} as $q\geq k$.

 Varolin suggested that one can take  $Y$ as a generic torus of dimension $\geq 2$.
In fact, a torus admits no divisors  if and only if it has algebraic dimension zero, i.e., the only meromorphic functions are constant
(e.g.,  \cite[p.\ 31]{EF82}). Of course,  any compact complex manifold with algebraic dimension zero admits no divisors, since $H^{0}\left(X, \mathcal{K}_{X}^{*} / \mathcal{O}_{X}^{*}\right) \cong \operatorname{Div}(X)$.  On the other hand,
for a very general lattice $\Gamma \subset \mathbb{C}^{n}$ the meromorphic function field $\mathcal{K}\left(\mathbb{C}^{n} / \Gamma\right)$ is trivial (e.g.,  \cite[p.\ 58]{Huy05}). By the way,  one can also take  $Y$ as any simple manifold (e.g.,  $K3$ surfaces or the general member of the deformation families
of hyperk\"ahler manifolds), since the  algebraic dimension of a simple manifold always vanishes (e.g.,  \cite[p. 132]{CDV14}).
\qed
\end{ex}

Just as \cite[p. 423]{M-V19}, if $\eta$ is an $L$-valued $(0, q)$-form on $Y,$ the orthogonal projection $$P: T_{X}^{0,1} |_Y \rightarrow T_{Y}^{0,1}$$ induced by the K\"ahler metric $\omega$ maps $\eta$ to the ambient $L$-valued $(0, q)$-form $P^{*} \eta$, given by
$$\left\langle P^{*} \eta, \bar{v}_{1} \wedge \cdots \wedge \bar{v}_{q}\right\rangle:=\left\langle\eta,\left(P \bar{v}_{1}\right) \wedge \cdots \wedge\left(P \bar{v}_{q}\right)\right\rangle$$
in $L_{y}$
for all $y \in Y\stackrel{\iota}{\hookrightarrow} X$ and $v_{1}, \ldots, v_{q} \in T_{X, y}^{ 1,0}$. Around $y$, we choose a local coordinate and frame $(U,  \{z_1,\ldots, z_n\},\sigma)$ such that
$Z\cap U=\{z_1=0\}$ and $\{d\bar{z}_1, d\bar{z}_2,\cdots,d\bar{z}_n\}$ is an $\omega(y)$-orthonormal basis of $\wedge^{0,1}T^{\star}_{X,y}$. At $y$,
set $\eta=\sum_{1\notin J}a_J d\bar{z}_J\circ \iota\otimes\sigma$ and then $P^{*}\eta=\sum_{1\notin J}a_J d\bar{z}_J\otimes\sigma$. So  $P^{*}$  is an isometry for the pointwise norm of $L$-valued $(0, q)$-forms induced by $\omega$ and the metric of $L$. As $\iota^{*} P^{*} \eta=\eta,$ we can apply Theorem  \ref{ambient extension} to $f=P^{*} u$ and obtain Theorem \ref{intrinsic extension}, while a sketch of a direct proof for it is also given in Appendix \ref{app}.
\begin{thm}[Intrinsic $L^2$ extension]\label{intrinsic extension}
With the setting of Theorem \ref{ambient extension}, there is a universal constant $C > 0$ such that for any smooth
$\bar\partial$-closed $(K_X\otimes L)|_Y$-valued $(0, q)$-form $u$ on $Y$ satisfying
\[
\int_Y\frac{|u|^2_{\omega} e^{-\varphi}}{|ds|^2_{\omega}e^{-\lambda}}dV_{Y,\omega}<\infty,
\]
  there exists a smooth $\bar\partial$-closed $K_X\otimes L$-valued $(0, q)$-form $U$ on $X$ with
$$\iota^* U=u \quad \text{and} \quad \int_X|U|^2_\omega e^{-\varphi}dV_{X,\omega} \leq \frac{C}{\delta}\int_Y\frac{|u|^2_{\omega} e^{-\varphi}}{|ds|^2_{\omega}e^{-\lambda}}dV_{Y,\omega}<\infty .$$
\end{thm}

We now present two applications of the main theorems. The first one is a  surjectivity theorem for the restriction maps in Dolbeault cohomology. It is similar to the extension theorems for cohomology
classes (without $L^2$ estimate) recently obtained by Cao--Demailly--Matsumura \cite[Theorem 1.1]{CDM17} and Zhou--Zhu  \cite[Theorem 1.1, Remark 1.1]{ZZ19}   on a holomorphically convex manifold with the more general curvature conditions, respectively, while our method is rather different from theirs.
\begin{cor}\label{restr}
 Let $X, Y, E, L$ be as in  Theorem \ref{intrinsic extension} and also $Y$  compact. Then the restriction morphism
\[
H^{0,q}\left(X, K_{X} \otimes L\right) \longrightarrow H^{0,q}\left(Y,\left(K_{X} \otimes L\right)_{| Y}\right)
\]
is surjective for any $0\leq q \leq n-1$.
\end{cor}
%By the way, in the language of sheaf theory, Corollary \ref{restr} tells us that the homomorphism
%\[
%H^{q}\left(X, \mathcal{O}_{X}\left(K_{X} \otimes L\right)\right) \rightarrow H^{q}\left(Y, \mathcal{O}_{Y}\left(K_{X} \otimes L\right)\right)\cong H^{q}\left(X, \mathcal{O}_{X}\left(K_{X} \otimes L\right) \otimes \mathcal{O}_{X} / \mathcal{J}\right),
%\]
%where $\mathcal{J}$ is the ideal sheaf of $Y$,
%is surjective by Leray's isomorphism. The surjectivity is of course equivalent to the injectivity of the homomorphism
%\[
%H^{q+1}\left(X, \mathcal{O}_{X}\left(K_{X} \otimes L\right) \otimes \mathcal{J}\right) \rightarrow H^{q+1}\left(X, \mathcal{O}_{X}\left(K_{X} \otimes L\right)\right),
%\]
%for any fixed $0\leq q \leq n-1$.
%Note that Corollary \ref{restr} is trivial when  $L$ is highly positively curved while it may still satisfy the setting of
%Theorem \ref{intrinsic extension}.

Much inspired by \cite[Corollary 4.11]{Dm00}, we consider the extension behavior of $\dbar$-closed forms on bounded pseudoconvex domains as the second application of the main theorems. It tells us some information about the relationship between  smooth ``generalized quasi-plurisubharmonic" functions and $\dbar$-closed forms.
\begin{cor}\label{cor2}
Let $\Omega \subset \mathbb{C}^{n}$ be a bounded  pseudoconvex domain, and $Y \subset \Omega$ a smooth hypersurface defined by a section $s$ of some Hermitian holomorphic line bundle $(E,e^{-\lambda})$ over $\Omega$.  Assume that $s$ is everywhere transverse to the zero section and that $|s| \leqslant 1$ on $\Omega$. Let $\varphi$ be an any smooth function such that for $0\leq q \leq n-1$, some
 $\delta\in (0,1]$ and the Chern curvature form $\Theta_E$ of $(E,e^{-\lambda})$, $\sqrt{-1}\partial\dbar\varphi-\Theta_E$ and $\sqrt{-1}\partial\dbar\varphi-(1+\delta)\Theta_E$ are below-bounded in the sense of \eqref{pos} on $\Omega$,
 i.e., the sums of their smallest $q+1$ eigenvalues with respect to the standard complex Euclidean metric are below bounded on $\Omega$. Then there is a constant
$C>0$ (depending  on $\Omega$ and the ``lower bound" of  $\sqrt{-1}\partial\dbar\varphi-\Theta_E$ and $\sqrt{-1}\partial\dbar\varphi-(1+\delta)\Theta_E$
on $\Omega$  in the sense of \eqref{pos}), with the following property: for any   $\dbar$-closed $(n,q)$-form or $(0,q)$-form $f$ on $Y$ with
$$\int_{Y}|f|^{2}\left|d s\right|^{-2} e^{-\varphi} d V_{Y}<+\infty,$$
 there exists a $\dbar$-closed  extension $F$ of $f$ to $\Omega$ with
\[
\int_{\Omega} |F|^{2}e^{-\varphi} d V_{\Omega} \leqslant C \int_{Y} \frac{|f|^{2}}{\left|d s\right|^{2}} e^{-\varphi} d V_{Y}.
\]
\end{cor}
\begin{proof}
Assume first that $f$ is a $\dbar$-closed $(n,q)$-form   on $Y$. Let $L:=\Omega \times \mathbb{C}$ be the trivial bundle equipped with a metric $e^{-\varphi-A|z|^{2}}$. We can choose a sufficiently large constant $A>0$ which depends on   the ``lower bound" of  $\sqrt{-1}\partial\dbar\varphi-\Theta_E$ and $\sqrt{-1}\partial\dbar\varphi-(1+\delta)\Theta_E$
on $\Omega$  in the sense of \eqref{pos} such that the curvature assumptions \eqref{assone} and \eqref{asstwo} are satisfied. Then there exists an extension $F$ of $f$ to $\Omega$ such that
\[
\int_{\Omega} |F|^{2}e^{-\varphi} e^{-A|z|^2} d V_{\Omega} \leqslant \widehat{C} \int_{Y} \frac{|f|^{2}}{\left|d s\right|^{2}} e^{-\varphi}e^{-A|z|^2} d V_{Y}
\]
according to  Theorem \ref{intrinsic extension}.
Note that $e^{-A|z|^2}$ has  lower and upper bounds which depend on $\Omega$. So there exists $C$ which depends on $\Omega$ and the ``lower bound" of  $\sqrt{-1}\partial\dbar\varphi-\Theta_E$ and $\sqrt{-1}\partial\dbar\varphi-(1+\delta)\Theta_E$
on $\Omega$  in the sense of \eqref{pos}, such that
\[
\int_{\Omega} |F|^{2}e^{-\varphi} d V_{\Omega} \leqslant C \int_{Y} \frac{|f|^{2}}{\left|d s\right|^{2}} e^{-\varphi} d V_{Y}.
\]
When  $f$ is a $\dbar$-closed $(0,q)$-form   on $Y$,  the application of the above argument for $f\wedge dz_1\wedge dz_2\wedge\cdots \wedge dz_n$ and $|dz_1\wedge dz_2\wedge\cdots \wedge dz_n|=1$ (possibly after normalizing the Euclidean metric) complete the proof.
\end{proof}

One would expect to weaken the pseudoconvexity assumption of $X$ in Theorem \ref{ambient extension} as the existence of an upper semi-continuous exhaustion on $X$ and thus $\Omega$ in Corollary \ref{cor2} could be weakened to be just a bounded domain in $\mathbb{C}^{n}$ admitting an upper semi-continuous exhaustion.
However, Varolin provided us with a  counterexample to the expectation.

\begin{ex}\label{ex2}
Set the domain $\Omega:=\left\{z \in \mathbb{C}^{2};  \frac{1}{2}<|z|<1\right\}$ and the subspace
$$Y:=\left\{z=\left(z^{1}, z^{2}\right) \in \Omega ; z^{2}=0\right\}.$$
Then it is easy to construct a continuous exhaustion, such as $\rho=\frac{1}{(|z|-1/2)(1-|z|)}$. Take the function $f(\zeta, 0):=\zeta^{-n}$ for any  $n\in \mathbb{N}^+$. Then $\int_{Y}|f|^{2}|d z^2|^{-2} d V_Y<+\infty$. So if one assumes the above expectation, then there exists some $F \in \mathcal{O}(\Omega)$ such that $\left.F\right|_{Y}=f$.  By Hartogs theorem and the identity theorem,  $F$ has a unique extension to the unit ball.  The restriction of this extension to $Y$ agrees with $f,$ and this means that $f$ itself has a holomorphic extension to the unit disk $\mathbb{D} \times\{0\}$. This is impossible by the identity theorem and that $\zeta^{-n}$ blows up near the origin in $\mathbb{D}$.\qed
\end{ex}

%A further natural question relating to this present paper  is that can we   relax the  condition on the rank of   $E$ in Theorem \ref{ambient extension} to the
% higher rank case.

A further interesting topic about the extension of $\bar\partial$-closed forms is the singular metric version of the main results here, which is very attractive and full of application prospects, e.g.,  the deformation invariance problem of higher cohomology of the pluricanonical bundle of  a  K\"ahler family.  However, it seems very difficult for extensions of general $\dbar$-closed forms.  In \cite[Remark 1.2]{M-V19}, McNeal--Varolin told us that the routine method---taking a regularization of the singular weight first and then passing to some kind of limit---cannot get the singular version of their extension theorems at least on a Stein manifold due to that the minimal extension operator may not exist.

The other difficulty of dealing with the singular metric version is  that  the operator $\dbar$-Laplacian with respect to a singular metric may lose the ellipticity in general. Demailly \cite[p.\ 17]{Dm00} hoped that the Laplacian with respect to a singular metric may have a little ``ellipticity" when the singularity of the metric involved is mild. The expectation of Demailly seems to be a difficult problem in PDE. All in all,  it seems very difficult to obtain
 the singular metric version of an $L^2$ extension theorem of general $\dbar$-closed forms.

 The  paper is organized as follows. We list some results  in Section \ref{pre} to be used in the
proof of Theorem \ref{ambient extension}. Then, we  prove Theorem \ref{ambient extension} in Section \ref{proof}. At last,
 we give a sketch of a direct proof of Theorem \ref{intrinsic extension} in Appendix \ref{app}.

\begin{notation}
Unless otherwise stated,  we will always adopt the notations in Section \ref{intro} in the latter sections and in particular,
use $|s|$ or $|s|e^{-\lambda/2}$  to denote the pointwise norm of $s$.
\end{notation}

\section{Some results used in the proofs}\label{pre}
In this section, we collect several results to be used in the proofs
of our main results.
Let $\iota :Y \hookrightarrow X$ be the natural inclusion of a smooth complex hypersurface $Y$ in a complex manifold $X$ and $L$ a  line bundle on $X$.

It is noteworthy that when $q\geq 1$, there are two natural choices for the restriction to  $Y$ of an $L$-valued $(0,q)$-form on  $X$:
\begin{enumerate}[(i)]
    \item
one can pull back the $L$-valued differential form on $X$ via the natural inclusion $Y\stackrel{\iota}{\hookrightarrow} X$ to produce an
$L$-valued $(0,q)$-form on $Y$, i.e., a section of $L|_Y\otimes \Lambda^{0,q}T^\star_Y\rightarrow Y$, which we call the \emph{intrinsic restriction}, or
    \item
one can view an $L$-valued $(0,q)$-form as an abstract section of an abstract bundle $L\otimes\Lambda^{0,q}T^\star_X$ and pullback the  section. That is to say, the restriction is a section of the restricted vector bundle $(L\otimes \Lambda^{0,q}T^\star_X)|_Y\rightarrow Y$. We call a section  of the vector bundle $(L\otimes \Lambda^{0,q}T^\star_X)|_Y\rightarrow Y$  an \emph{ambient
form}, and the restriction of an $L$-valued $(0,q)$-form on $X$ to $Y$ an \emph{ambient restriction}.
\end{enumerate}
More precisely, one has the following definitions.
\begin{defn}[{\cite[Definition 3.1]{M-V19}}] \label{Two}\
\begin{enumerate}[(i)]
\item An $L|_Y$-valued $(0,q)$-form $\eta$ on $Y$ is called the \emph{intrinsic restriction} of an $L$-valued $(0,q)$-form $\theta$ on $X$ if
\[
\iota ^*\theta = \eta.
\]
\item A section $\xi$ of the vector bundle $(L \otimes \Lambda ^{0,q}T^\star_X)|_Y$ is called the \emph{ambient restriction} of an $L$-valued $(0,q)$-form $\theta$ on $X$ if
\[
\theta (y) = \xi(y)
\]
for all $y \in Y$, that is $\theta|_Y=\xi$. Note that  $|\theta (y)| = |\xi(y)|$ on $Y$ when $X$ and $L$ are equipped with some Hermitian metrics in this case.
\end{enumerate}
\end{defn}
Since $(L\otimes \Lambda^{0,q}T^\star_X)|_Y$ is not a holomorphic vector bundle, it does not admit a natural notion of $\bar\partial$. Then one needs:
\begin{defn}[{\cite[p. 422]{M-V19}}]\label{pullback}
Let $\iota :Y \hookrightarrow X$ be the natural inclusion of a smooth complex hypersurface $Y$ in a complex manifold $X$ and $L$ a  line bundle on $X$.
Then for any  ambient form $\xi$,
 $\iota^{*} \xi$ is defined to be an $L|_Y$-valued
$(0, q)$-form on $Y$ by
$$
\left\langle \iota^{*} \xi, \bar{v}_{1}, \ldots, \bar{v}_{q}\right\rangle:=\left\langle\xi, d \iota(y) \bar{v}_{1}, \ldots, d \iota(y) \bar{v}_{q}\right\rangle \quad \text { in } L_{y}\ (y\in Y), \quad v_{1}, \ldots, v_{q} \in T_{Y, y}^{1,0}.
$$
Note that $\dbar(\iota^{*}\xi)$ is now well defined naturally.
\end{defn}

In \cite{M-V19}, McNeal--Varolin established the following two extension theorems and we will use them to construct some smooth extensions locally on a weakly pseudoconvex K\"ahler manifold in Subsection \ref{sect0.1} and Appendix \ref{app}.
\begin{thm}[Ambient $L^{2}$ extension]\label{mv Ambient}
 Let $X$ be an $n$-dimensional \emph{Stein} manifold with K\"ahler form $\omega$ and $\iota: Y \hookrightarrow X$ a smooth hypersurface. Let $L \rightarrow X$ be a holomorphic line bundle with smooth Hermitian metric $e^{-\varphi}$. Assume that the line bundle $L_{Y} \rightarrow X$
associated to the smooth divisor $Y$ has a section $s \in H^{0}\left(X, L_{Y}\right)$ and a smooth
Hermitian metric $e^{-\lambda}$ such that $Y$ is the divisor of $s$ and
\[
\sup _{X}\left|s\right|^{2} e^{-\lambda} \leq 1.
\]
Assume also that for any $0\leq q \leq n-1$,
\[
\sqrt{-1}\left(\partial \bar{\partial}\left(\varphi-\lambda\right)+\operatorname{Ricci}(\omega)\right) \wedge \omega^{q} \geq 0
\]
and
\[
\sqrt{-1}\left(\partial \bar{\partial}\left(\varphi-(1+\delta) \lambda\right)+\operatorname{Ricci}(\omega)\right) \wedge \omega^{q} \geq 0
\]
for some constant $0<\delta\leq 1.$ Then there is a constant $C>0$ such that for any smooth section $\xi$ of the vector bundle $\left.\left(L \otimes \Lambda^{0, q}T^{\star}_{X}\right)\right|_{Y} \rightarrow Y$ satisfying
\[
\bar{\partial}\left(\iota^{*} \xi\right)=0 \quad \text {and} \quad \int_{Y} \frac{|\xi|_{\omega}^{2} e^{-\varphi}}{\left|d s\right|^{2}_{\omega} e^{-\lambda}} dV_{Y,\omega}<+\infty,
\]
there exists a smooth $\dbar$-closed $L$-valued $(0, q)$-form $\Xi$ on $X$ with
\[
\left.\Xi\right|_{Y}=\xi \quad \text { and } \quad \int_{X}|\Xi|_{\omega}^{2} e^{-\varphi} dV_{\omega} \leq \frac{C}{\delta} \int_{Y} \frac{|\xi|_{\omega}^{2} e^{-\varphi}}{\left|d s\right|^{2}_{\omega} e^{-\lambda}} dV_{Y,\omega}.
\]
The constant $C$ is universal, i.e., it is independent of all the data.
\end{thm}

\rem On the above theorem, there is a typo on \cite[Theorem 1]{M-V19} which only requires $\delta>0$. In fact, we can  conclude that the first two lines    in \cite[p.\ 438]{M-V19} cannot be true if $\delta>1$. Denote $e^{v}$ in  \cite[p.\ 438]{M-V19} by $|s|^2$ here, and then
 $$e^{v}(\tau+A)=|s|^2(2+2e^{a-1}+\log(2e^{a-1}-1))\geq |s|^2\cdot 2e^{a-1}=\frac{2e^{\gamma-1}|s|^2}{(\varepsilon^2+|s|^2)^{\delta}}.$$
As $\delta>1$, the above is
 $$2e^{\gamma-1}\cdot\frac{|s|^2}{\varepsilon^2+|s|^2}\cdot\frac{1}{(\varepsilon^2+|s|^2)^{\delta-1}},$$
 which cannot be bounded when $\varepsilon\to 0$ and  $|s|\to 0$ (take $|s|=\varepsilon\to 0$ for example).

\begin{thm}[Intrinsic $L^{2}$ extension]\label{mv Intrinsic}
With the hypotheses of Theorem \ref{mv Ambient}, there is a universal constant $C>0$ such that for any smooth
$\dbar$-closed $L$-valued $(0, q)$-form $\eta$ on $Y$ satisfying
\[
\int_{Y} \frac{|\eta|_{\omega}^{2} e^{-\varphi}}{\left|d s\right|^{2}_{\omega} e^{-\lambda}} dV_{Y,\omega}<+\infty,
\]
there exists a smooth $\bar{\partial}$-closed $L$-valued $(0, q)$-form $\Pi$ on $X$ such that
\[
\iota^{*} \Pi=\eta \quad \text{and} \quad \int_{X}|\Pi|_{\omega}^{2} e^{-\varphi} dV_{\omega} \leq \frac{C}{\delta} \int_{Y} \frac{|\eta|_{\omega}^{2} e^{-\varphi}}{\left|d s\right|^{2}_{\omega} e^{-\lambda}} dV_{Y,\omega}.
\]
\end{thm}

Consider the modified $\bar{\partial}$-operators
$$T:=\bar{\partial}\circ\sqrt{\tau+A}\quad \text{and}\quad  S:=\sqrt{\tau}\cdot\bar{\partial}$$
acting on $(n, q)$-forms with values in a vector bundle, where  $\tau$, $A$ are positive smooth functions.
%and $\dbar$ is taken as the usual maximal extension of $\dbar$ acting on smooth forms. It is easy to see that they are closed and densely defined when the metrics involved are smooth.
Then $S \circ T=0$. In solving $\dbar$-equation,  the  basic estimate about the modified $\bar{\partial}$ operator is always needed to construct some bounded linear functionals.  On different occasions, several classical    basic estimates have been established in \cite{O95, B96, M96, S96, Dm00}, for example. Here we adopt the following one.
\begin{lemma}[Twisted basic estimate]\label{basic estimate}
Let $(X, \omega)$ be a  K\"ahler manifold  and $E$ a holomorphic line  bundle with a smooth Hermitian
metric $e^{-\varphi}$ over $X$. Assume that $\tau$ and $A$ are smooth and  positive functions on $X$. Fix a smoothly bounded domain $\Omega \subset \subset X$ such that its boundary $\partial \Omega$ is pseudoconvex (e.g.,  \cite[Section 1.5]{Va10} for this notion and its effect).
Then for any smooth  $E$-valued $(n, q)$-form $u$ in the domain of $\dbar_{\varphi}^{*}$, one has the estimate
$$\begin{aligned}
& \int_{\Omega}(\tau+A)\left|\dbar_{\varphi}^* u\right|_{\omega}^{2}e^{-\varphi} d V_{\omega}+\int_{\Omega} \tau\left|\dbar u\right|_{\omega}^{2}e^{-\varphi} d V_{\omega} \\
\geq & \int_{\Omega}\left\langle\left[\sqrt{-1}\left(\tau \partial \bar{\partial}\varphi-\partial \bar{\partial} \tau-\frac{\partial \tau \wedge \bar{\partial} \tau}{A}\right), \Lambda_\omega\right] u, u\right\rangle_{\omega} e^{-\varphi} d V_{\omega},
\end{aligned}$$
where $\Lambda_\omega$ is the dual Lefschetz operator.
\end{lemma}

\begin{proof}
The proof is the same as that of \cite[Lemma 2.2]{M-V19} which is in the context of  the Stein setting. However, the estimate attributes essentially to the usual  twisted Bochner--Kodaira--Morrey--Kohn identity and the pseudoconvexity of $\partial\Omega$.
\end{proof}

Furthermore, the ellipticity of the twisted $\dbar$-Laplacian $\square:=T T^{*}+S^{*} S$ is needed.
\begin{lemma}[{\cite[Proposition 2.3]{M-V19}}]\label{ellipticity}
 Assume that the functions $\tau$, $A$ and the Hermitian metric $e^{-\varphi}$ of the holomorphic line bundle $E$  are smooth, and that $\tau$ and $\tau+A$ are positive, then the operator $\square$ is second order (interior) elliptic with smooth coefficients.
\end{lemma}

\begin{lemma}[{\cite[Lemma 6.9]{Dm82}}]\label{extend across}
Let $\Omega$ be an open subset of $\mathbb{C}^{n}$ and $Y$ an analytic subset of $\Omega$.  Assume that $v$ is a $(p, q-1)$-form with $L_{\text {loc }}^{2}$ coefficients and $w$ is a $(p, q)$-form with $L_{\text {loc }}^{1}$ coefficients such that $\dbar v=w$ on $\Omega\setminus Y$ (in the sense of distribution theory). Then $\dbar v=w$ on $\Omega$. (A more general version for the first order differential operator can be found in \cite[Proposition 4.8]{B18}.)
\end{lemma}

\begin{lemma}[{\cite[Chapter VI-(5.8) Proposition]{Dm12}}]\label{eigen}
Let $(X, \omega)$ be an $n$-dimensional  Hermitian manifold and  $\gamma$  a real $(1,1)$-form.
Then there exists an $\omega$-orthogonal basis $\left(\zeta_{1}, \zeta_{2}, \ldots, \zeta_{n}\right)$ in $T_{X}^{1,0}$ which diagonalizes both forms $\omega$ and $\gamma:$
\[
\omega=\sqrt{-1} \sum_{1 \leqslant j \leqslant n} \zeta_{j}^{\star} \wedge \bar{\zeta}_{j}^{\star}, \quad \gamma= \sqrt{-1}\sum_{1 \leqslant j \leqslant n} \gamma_{j} \zeta_{j}^{\star} \wedge \bar{\zeta}_{j}^{\star}, \quad \gamma_{j} \in \mathbb{R}.
\]
For every form $u=\sum u_{J, K} \zeta_{J}^{\star} \wedge \bar{\zeta}_{K}^{\star},$ one has
\[
[\gamma, \Lambda_\omega] u=\sum_{J, K}\left(\sum_{j \in J} \gamma_{j}+\sum_{j \in K} \gamma_{j}-\sum_{1 \leqslant j \leqslant n} \gamma_{j}\right) u_{J, K} \zeta_{J}^{\star} \wedge \bar{\zeta}_{K}^{\star}.
\]
\end{lemma}

\begin{lemma}[{\cite[Chapter VIII-(6.4)]{Dm12}}]\label{less than q eigen}
Let $\theta$ be a smooth real $(1,1)$-form on an $n$-dimensional Hermitian manifold $(X, \omega)$.
If $\lambda_{1}(x) \leqslant \cdots \leqslant \lambda_{n}(x)$ are the eigenvalues of $\theta$ with respect to $\omega$ for all $x \in X$ and $\lambda_{1}+\cdots+\lambda_{q}>0$, then for arbitrary $(n,q)$-form $g$ on $X$,
$$\left\langle [\theta,\Lambda_\omega] g, g\right\rangle_{\omega} \geqslant\left(\lambda_{1}+\cdots+\lambda_{q}\right)|g|_{\omega}^{2}$$
and
$$\int_{X}\left\langle [\theta,\Lambda_\omega]^{-1} g, g\right\rangle_{\omega} d V_{\omega} \leqslant \int_{X} \frac{1}{\lambda_{1}+\cdots+\lambda_{q}}|g|_{\omega}^{2} dV_{\omega}.$$
\end{lemma}

\begin{lemma}\label{Twedge}
 Let $(X,\omega)$ be an $n$-dimensional Hermitian manifold and  $\theta$  a continuous $(1,0)$-form. Then for any $(n, q)$-form $\alpha$, we have
\[
\left[\sqrt{-1} \theta \wedge \bar{\theta}, \Lambda_{\omega}\right] \alpha=T_{\bar{\theta}} T_{\bar{\theta}}^* \alpha,
\]
where $T_{\bar{\theta}}$ denotes $\bar{\theta}\wedge\bullet$.
\end{lemma}

\begin{proof}
The proof is the same as that of \cite[Lemma 4.2]{GZ15b}.
\end{proof}

From classical knowledge about  matrices (e.g.,  \cite{L07}), we can easily get the following  result.
\begin{lemma}\label{matrix}
Assume that $A$ and $B$ are Hermitian matrices of $n\times n$ and both of them are positive definite. Then $A-B>0$ implies that $B^{-1}-A^{-1}>0$.
\end{lemma}

From Lemma \ref{matrix}, we can easily conclude the following comparison theorem.

\begin{lemma}\label{inversion of order}
Let $\theta_1$ and $\theta_2$ be  smooth real $(1,1)$-forms on an $n$-dimensional Hermitian manifold $(X, \omega)$. Assume that $(\theta_1-\theta_2)\wedge \omega^r>0$, $\theta_1\wedge \omega^r>0$ and $\theta_2\wedge \omega^r>0$. Then
$$[\theta_2,\blambda_{\omega}]^{-1}-[\theta_1,\blambda_{\omega}]^{-1}$$ is positive definite on $(n,r+1)$-forms.
\end{lemma}

\rem The above lemma is  a little bit different from and stronger than the classical result about the non-increasing of $[\theta,\blambda_{\omega}]^{-1}$ with respect to $\theta$ (see  \cite[Lemma 3.2]{Dm82} or \cite[Propositon 5.2]{B18}). Here we know nothing about the  comparison information between $\theta_1$ and $\theta_2$ and only know the  comparison data between  $\theta_1\wedge \omega^r$ and $\theta_2\wedge \omega^r$. However, from a point of view of positive definite  transformation we can easily get the above lemma.

The following observation is useful to give a direct proof of Proposition \ref{is ambient extension}.
\begin{lemma}[{\cite[p. 609]{GZ15a}, or \cite[Lemma 9.20]{B18}}]\label{integrability}
If $g$ is an integrable function near $0 \in \mathbb{R}^{d},$ then there exists a sequence $x_{j} \rightarrow 0$ in $\mathbb{R}^{d}$ such that $\left|g\left(x_{j}\right)\right|=o\left(\left|x_{j}\right|^{-d}\right)$.
\end{lemma}

\section{Proof of Theorem \ref{ambient extension}}\label{proof}
We  split the proof of Theorem \ref{ambient extension} in several steps.
\subsection{Construction of a smooth extension $\widetilde{f}_{\infty}$}\label{sect0.1}
Let \{$W_{\alpha}$\} be the Stein coordinate patches  of $X$, biholomorphic to polydiscs, and admit the following property: if we denote the corresponding coordinates by $\left(z_{\alpha}, w_{\alpha}\right) \in \Delta \times \Delta^{n-1},$ where $w_{\alpha}=\left(w_{\alpha}^{1}, \ldots, w_{\alpha}^{n-1}\right),$ then $W_{\alpha} \cap Y=\left\{z_{\alpha}=0\right\}$.
On each $W_{\alpha},$ we fix some holomorphic $\sigma_{\alpha} \in \Gamma\left(W_{\alpha}, K_{X} \otimes L\right)$ to trivialize
$K_{X} \otimes L$. Let  \{$\theta_{\alpha}$\} be a partition of unity
subordinate to \{$W_{\alpha}$\}. For arbitrary $\alpha$, \eqref{assone},\eqref{asstwo},\eqref{assthree} and $-\sqrt{-1}\operatorname{Ricci}(\omega)=\Theta(K_X)$ (the Chern curvature form of $K_X$) along with Theorem \ref{mv Ambient} can conclude that there exists an ambient $\dbar$-closed extension $f_{\alpha}$
of $f$ on $W_{\alpha}$. Set $\widetilde{f}_{\infty}: = \sum_{\alpha}\theta_{\alpha}\cdot f_{\alpha}$. Then
$$\dbar\widetilde{f}_{\infty}=\dbar\sum_{\alpha}\theta_{\alpha}\cdot (f_{\alpha}-f_{\beta}) = \sum_{\alpha}\dbar\theta_{\alpha}\cdot(f_{\alpha}-f_{\beta})\text { on } W_{\beta}.$$
So $\dbar\widetilde{f}_{\infty} = 0$ along $Y$. Note that $\widetilde{f}_{\infty}$ can also be obtained by the same method as \cite[(4.4)]{Dm00} or \cite[Proof of Lemma 3.1]{Koz11}.

For the nice regularity  of $g_{\vep}$ in Subsection \ref{subs-3.3}, we need some regularity of $\dbar\widetilde{f}_{\infty}$ twisted by $s^{-1}$. On  this, Koziarz \cite[Lemma 3.1]{Koz11} had provided us with
a useful trick to raise the regularity of $s^{-1}\dbar\widetilde{f}_{\infty}$ and still preserve the ambient extension property of $\widetilde{f}_{\infty}$  as follows.

We proceed by induction to get the regularity of $s^{-1}\dbar \widetilde{f}_{\infty}$. One says that $\widetilde{f}_{\infty}$ enjoys \emph{the property $\left(P_{k}\right)$} for $k \geq 1$, if on each $W_{\alpha}$
\begin{equation}\label{lifting}
\begin{aligned}
\dbar \widetilde{f}_{\infty}=&\ z_{\alpha} \widetilde{f}_{\alpha}\left(z_{\alpha}, w_{\alpha}\right)+\bar{z}_{\alpha}^{k}\left[d \bar{z}_{\alpha} \wedge \sum_{|I|=q} a_{I}\left(w_{\alpha}\right) \sigma_{\alpha} d \bar{w}_{\alpha}^{I}+\sum_{\left|I^{\prime}\right|=q+1} b_{I^{\prime}}\left(w_{\alpha}\right) \sigma_{\alpha} d \bar{w}_{\alpha}^{I^{\prime}}\right] \\
&+\bar{z}_{\alpha}^{k+1} h_{\alpha}\left(z_{\alpha}, w_{\alpha}\right)
\end{aligned}
\end{equation}
for some $\widetilde{f}_{\alpha}, h_{\alpha} \in \mathcal{E}^{\infty}\left(W_{\alpha}, \Lambda^{n, q+1} T_{X}^{\star} \otimes L\right)$ and  $a_{I}, b_{I^{\prime}} \in \mathcal{E}^{\infty}\left(\Delta^{n-1}, \mathbb{C}\right)$
with the increasing multi-indices $I, I^{\prime}$.

Note that for $k\geq 2$, $\frac{\bar{z}^k}{z}$ is of class $\mathcal{E}^{k-2}$ on a complex plane, where $z$ is the complex analytic coordinate.
 Then $\widetilde{f}_{\infty}$ enjoying \emph{the property $\left(P_{k}\right)$}  implies that
$$s^{-1} \dbar \widetilde{f}_{\infty} \in \mathcal{E}^{k-2}\left(X, \Lambda^{n,q+1} T_{X}^{\star} \otimes\right.\left.L \otimes \mathcal{O}_{X}(-Y)\right),\quad \text{for $k \geq 2$}.$$

In  mathematical analysis, one has an observation in the spirit of the Taylor expansion that, a smooth function $f$ of one real variable  which vanishes at the origin can be written as $f(x)=x\cdot g(x)$ for a smooth function $g$. It can be proved by the L'Hospital's rule to show that $g$ is of class $\mathcal{E}^{k}$ for any positive $k$.
Then since  $\dbar\widetilde{f}_{\infty}=0$ along $Y\cap {W_{\alpha}}=\{z_{\alpha}=0\}$ and $\dbar\widetilde{f}_{\infty}$ is smooth, one can   expand $\dbar\widetilde{f}_{\infty}$ along the coordinate function $z_{\alpha}$ such that
$$\dbar\widetilde{f}_{\infty}=z_{\alpha}\cdot g_1+\bar{z}_{\alpha}\cdot g_2,$$
where $g_1$ and $g_2$ are the corresponding smooth forms. So expanding $g_2$  along the coordinate function $z_{\alpha}$ again similarly can conclude that $\widetilde{f}_{\infty}$ enjoys \emph{the property $\left(P_{1}\right)$}.

A direct calculation shows
$$
\begin{aligned}
\dbar\left(\dbar \widetilde{f}_{\infty}\right)=&\ z_{\alpha} \dbar \widetilde{f}_{\alpha}\left(z_{\alpha}, w_{\alpha}\right)+k \bar{z}_{\alpha}^{k-1} d \bar{z}_{\alpha} \wedge\left[\sum_{\left|I^{\prime}\right|=q+1} b_{I^{\prime}}\left(w_{\alpha}\right) \sigma_{\alpha} d \bar{w}_{\alpha}^{I^{\prime}}\right] \\
&+\bar{z}_{\alpha}^{k} h_{\alpha}^{\prime}\left(z_{\alpha}, w_{\alpha}\right)
\end{aligned}
$$
for some $h_{\alpha}^{\prime} \in \mathcal{E}^{\infty}\left(W_{\alpha}, \Lambda^{n, q+2} T_{X}^{\star} \otimes L\right)$.  Then all the above $b_{I^{\prime}}$   vanish identically as  $\dbar\left(\dbar \widetilde{f}_{\infty}\right)=0$.
So we take
$$
\widetilde{f}_{\infty}^{\prime}=\widetilde{f}_{\infty}-\sum_{\alpha} \theta_{\alpha} \frac{\bar{z}_{\alpha}^{k+1}}{k+1} \sum_{|I|=q} a_{I}\left(w_{\alpha}\right) \sigma_{\alpha} d \bar{w}_{\alpha}^{I}
$$
to get
$$
\begin{aligned}
\dbar \widetilde{f}_{\infty}^{\prime}=&\ \dbar \widetilde{f}_{\infty}-\sum_{\alpha}\left(\frac{\bar{z}_{\alpha}^{k+1}}{k+1} \dbar \theta_{\alpha}+\theta_{\alpha} \bar{z}_{\alpha}^{k} d \bar{z}_{\alpha}\right) \wedge \sum_{|I|=q} a_{I}\left(w_{\alpha}\right) \sigma_{\alpha} d \bar{w}_{\alpha}^{I} \\
&\ -\sum_{\alpha} \theta_{\alpha} \frac{\bar{z}_{\alpha}^{k+1}}{k+1} \dbar\left(\sum_{|I|=q} a_{I}\left(w_{\alpha}\right) \sigma_{\alpha} d \bar{w}_{\alpha}^{I}\right) \\
=&\ \sum_{\alpha} \theta_{\alpha}\left(\dbar \widetilde{f}_{\infty}-\bar{z}_{\alpha}^{k} d \bar{z}_{\alpha} \wedge \sum_{|I|=q} a_{I}\left(w_{\alpha}\right) \sigma_{\alpha} d \bar{w}_{\alpha}^{I}\right)+\sum_{\alpha} \bar{z}_{\alpha}^{k+1} h_{\alpha}^{\prime \prime}\left(z_{\alpha}, w_{\alpha}\right) \\
=&\ \sum_{\alpha} z_{\alpha} \theta_{\alpha} \widetilde{f}_{\alpha}\left(z_{\alpha}, w_{\alpha}\right)+\sum_{\alpha} \bar{z}_{\alpha}^{k+1}\left(\theta_{\alpha} h_{\alpha}\left(z_{\alpha}, w_{\alpha}\right)+h_{\alpha}^{\prime \prime}\left(z_{\alpha}, w_{\alpha}\right)\right)
\end{aligned}
$$
for some $h_{\alpha}^{\prime \prime} \in \mathcal{E}_{0}^{\infty}\left(W_{\alpha}, \Lambda^{n,q+1} T_{X}^{\star} \otimes L\right)$.  Then  $\dbar\widetilde{f}_{\infty}^{\prime}$ satisfies \eqref{lifting} for $k+1$ on each $W_{\alpha}$, possibly after some transformations of coordinates. Then
$\widetilde{f}_{\infty}^{\prime}$ enjoys \emph{the property $\left(P_{k+1}\right)$}. Apparently, $\widetilde{f}_{\infty}^{\prime}$ is still the ambient extension of $f$ while $\dbar\widetilde{f}_{\infty}^{\prime}$ still vanishes along $Y$.

In conclusion, we have proved the following results.
\begin{prop}\label{tilde{f}}
For any $k\geq 0$, there exists a smooth section
$$\widetilde f_\infty\in \mathcal E^\infty(X,\Lambda^{n,q}T_X^\star\otimes L)$$
such that
\begin{enumerate}[{\rm (}a{\rm )}]
\item \label{finfty=f} $\widetilde f_\infty$ is the ambient extension of $f$,
\item $\bar\partial\widetilde f_\infty=0$ at every point of $Y$,
\item \label{tilde{f}.d}
$s^{-1} \bar\partial\widetilde f_\infty\in \mathcal {E}^k(X,\Lambda^{n,q+1}T^\star_ X \otimes L\otimes \mathcal{O}_{X}(-Y))=\mathcal {E}^k(X,\Lambda^{n,q+1}T^\star_ X \otimes L\otimes E^*)$.
\end{enumerate}
\end{prop}
From now on we fix $k\geq n+6$, where $n$ is the complex dimension of $X$.
\subsection{Construction of special weights and twist factors}\label{sub-3.2}
For the sake of completeness, we write the following constructions concretely, which  is almost the same as \cite[$\S$ 4.1]{M-V19}.

Set
\[
e^{-\psi}:=e^{-\varphi+\lambda}.
\]
Next we turn to the choices of the functions $A$ and $\tau$ as in \cite{Va08,M-V07} and more originally \cite{M-V07}.

Let
$$h(x):=2-x+\log \left(2 e^{x-1}-1\right), \quad v:=\log |s|^2 \quad \text{and} \quad a:=\gamma-\delta \log \left(|s|^2+\varepsilon^{2}\right),$$
where $0<\delta\leq 1$ is as in the main theorems, $x > 1$, and $\gamma>1$ is a real number such that $a>1$. Note that
\begin{equation}\label{ine gamma}
a\geq\gamma-\delta\log(1+\varepsilon^2)\geq\gamma-\delta\varepsilon^2
\end{equation}
due to \eqref{assthree}. That is to say, for the fixed $\gamma>1$, there exists $\vep_0>0$ such that $a-1$ has a positive lower bound which is independent of $\vep$ as $\vep<\vep_0$. It is easy to see that
$$\label{h'}
h'(x) = (2e^{x-1}-1)^{-1} \in (0,1) \quad \text{and} \quad h''(x) = \frac{-2e^{x-1}}{(2e^{x-1}-1)^2} < 0.
$$
Define
\[
\tau:=a+h(a) \quad \text { and } \quad A:=\frac{\left(1+h^{\prime}(a)\right)^{2}}{-h^{\prime \prime}(a)}.
\]
Then $A = 2e^{a-1}$. Furthermore, \eqref{ine gamma} gives
\begin{equation}\label{tau-h'}
\tau-(1+h'(a))>\sigma'
\end{equation}
for some constant $\sigma'>0$ which is independent of $\vep$ when $\vep<\vep_0$.
Moreover, these choices guarantee that
$$-\partial \bar{\partial} \tau-A^{-1} \partial \tau \wedge \bar{\partial} \tau=\left(1+h^{\prime}(a)\right)(-\partial \bar{\partial} a).$$
Finally, a straightforward calculation yields
\[
\begin{aligned}
-\partial \bar{\partial} a &=\delta \partial \bar{\partial} \log \left(e^v+\varepsilon^{2}\right) \\
&=\frac{\delta |s|^2}{|s|^2+\varepsilon^{2}} \partial \bar{\partial} v+\frac{4 \delta \varepsilon^{2} \partial|s| \wedge \bar{\partial}|s|}{\left(|s|^2+\varepsilon^{2}\right)^{2}} \\
&=-\delta \frac{|s|^2}{|s|^2+\varepsilon^{2}} \partial \bar{\partial} \lambda+\frac{4 \delta \varepsilon^{2} \partial|s| \wedge \bar{\partial}|s|}{\left(|s|^2+\varepsilon^{2}\right)^{2}},
\end{aligned}
\]
where the last equality follows from the Lelong--Poincar\'e equation and   $|s|^2[Y]=0$ due to $\operatorname{Supp}[Y]=Y$.

A direct calculation together with \eqref{assone}, \eqref{asstwo} and  \eqref{tau-h'} yields
\begin{equation} \label{adhoc}
  \begin{split}
 &\ \sqrt{-1}\left(\tau(\partial \bar{\partial} \psi) -\partial \bar{\partial} \tau-A^{-1} \partial \tau \wedge \bar{\partial} \tau\right) \wedge \omega^{q}\\
=&\ \sqrt{-1}\left(\tau\partial \bar{\partial}\left(\varphi-\lambda\right)+\left(1+h^{\prime}(a)\right)(-\partial \bar{\partial} a)\right) \wedge \omega^{q} \\
=&\ \left(\tau-\left(1+h^{\prime}(a)\right)\left(\frac{|s|^2}{|s|^2+\varepsilon^{2}}\right)\right) \cdot\sqrt{-1}\partial \bar{\partial}\left(\varphi-\lambda\right) \wedge \omega^{q} \\
&\ +\sqrt{-1}\left(1+h^{\prime}(a)\right) \frac{|s|^2}{|s|^2+\varepsilon^{2}}\left(\partial \bar{\partial}\left(\varphi-\lambda\right)-\delta \partial \bar{\partial} \lambda\right) \wedge \omega^{q}\\
&\ +\sqrt{-1}\left(1+h^{\prime}(a)\right)\left(\frac{4\delta\varepsilon^2\partial|s|\wedge\dbar|s|}{(|s|^2+\varepsilon^2)^2}\right)\wedge\omega^{q}\\
>&\ \sqrt{-1}\delta\left(\frac{4\varepsilon^2\partial|s|\wedge\dbar|s|}{(|s|^2+\varepsilon^2)^2}\right)\wedge\omega^{q}+\sigma'\sigma
\omega\wedge\omega^q\\
>&\ 0.
  \end{split}
\end{equation}

Set
$$B_{\vep}:=\sqrt{-1}\left(\tau(\partial \bar{\partial} \psi) -\partial \bar{\partial} \tau-A^{-1} \partial \tau \wedge \bar{\partial} \tau\right).$$
Then \eqref{adhoc} tells us that
\begin{equation}\label{B^1}
B_{\vep}\wedge\omega^q>\sqrt{-1}\delta\left(\frac{4\varepsilon^2\partial|s|\wedge\dbar|s|}{(|s|^2+\varepsilon^2)^2}\right)\wedge\omega^{q}
\end{equation}
and
\begin{equation}\label{B^2}
B_{\vep}\wedge\omega^q\geq \sigma'\sigma \omega\wedge\omega^q
\end{equation}
hold on $X$.

\subsection{Solving twisted $\dbar$-Laplace equations with estimates}\label{subs-3.3}
Recall that in Theorem \ref{ambient extension}, $f$ is a smooth section of the vector bundle $\left.\left(K_{X} \otimes L \otimes \wedge^{0, q} T_{X}^{\star}\right)\right|_{Y} \rightarrow Y$
 for any $0\leq q \leq n-1$ satisfying
\[
\bar{\partial}\left(\iota^{*} f\right)=0 \quad \text { and } \quad \int_{Y} \frac{|f|_{\omega}^{2} e^{-\varphi}}{|d s|_{\omega}^{2} e^{-\lambda}} d V_{Y, \omega}<\infty.
\]
And one has obtained an ambient extension $\widetilde{f}_{\infty}$ of $f$ in Proposition \ref{tilde{f}}. Let $0< c\ll 1$ and $\theta \in \mathcal{E}_{0}^{\infty}([0,+\infty))$ a cutoff function with values in $[0,1]$ such that $\left.\theta\right|_{[0, c]} \equiv 1$ and $\theta|_{[1,+\infty)}\equiv 0$
and $\left|\theta^{\prime}\right| \leq 1+2c$.  For $\varepsilon>0$,   define
\[
g_{\varepsilon}:=s^{-1} \dbar(\theta\left(\varepsilon^{-2}|s|^{2}\right) \widetilde{f}_{\infty})=s^{-1}\dbar(\theta(\varepsilon^{-2}|s|^{2}))\wedge
\widetilde{f}_{\infty}+s^{-1}\theta(\varepsilon^{-2}|s|^{2})\dbar \widetilde{f}_{\infty}.
\]
The first  term in $g_{\vep}$ can be easily written as
\begin{equation}\label{g^1}
g_{\vep}^{(1)}=\dbar|s|\wedge\frac{2|s|(\varepsilon^{-2}|s|^{2}+1)\theta'(\varepsilon^{-2}|s|^{2})}{|s|^{2}+\vep^2}s^{-1}\widetilde{f}_{\infty}.
\end{equation}
We also denote the second term $s^{-1}\theta(\varepsilon^{-2}|s|^{2})\dbar \widetilde{f}_{\infty}$ in the above expression of $g_{\vep}$ by $g_{\vep}^{(2)}$.

From the smoothness of $g_{\vep}^{(1)}$ and the regularity information of $g_{\vep}^{(2)}$ by  Proposition \ref{tilde{f}}.\eqref{tilde{f}.d}, we know that
$g_{\vep}\in \mathcal{E}^{k}\left(X, \Lambda^{n, q+1} T_{X}^{\star} \otimes L \otimes E^*\right)$. Moreover, $g_{\vep}$ is $\dbar$-closed outside $Y$ according to  the definition of $g_{\vep}$. So $g_{\vep}$ is $\dbar$-closed on $X$ due to the continuity of $\dbar g_{\vep}$.

Assume that $\phi$ is a smooth plurisubharmonic exhaustion function of the weakly pseudoconvex  K\"ahler manifold $X$.  Due to the Sard's theorem, we can always assume that $\bomega_j:=
\left\{\phi<j  \right\}, j=1, 2, \ldots,$ satisfy
\[
\Omega_{j} \subset \subset \Omega_{j+1} \quad \text { and } \quad \lim\limits _{j \rightarrow \infty} \Omega_{j}=\bigcup_{j \geq 1} \Omega_{j}=X,
\]
and every $\partial\Omega_{j}$ is smooth and pseudoconvex.
From now on, we will work on $\Omega_j$ instead of $X$ until the end of the penultimate step.

For any smooth  $K_X\otimes L\otimes E^*$-valued $(0, q+1)$-form $\beta$ in the domain of $T^*$ and $S$,
we infer the inequality
 \begin{equation} \label{C-S estimate}
  \begin{split}
 |\langle \beta, g_{\vep}\rangle|^{2}_{L^2}
 &\leq \langle [B_{\vep},\Lambda_{\omega}]^{-1}g_{\vep},g_{\vep} \rangle_{L^2} \langle [B_{\vep},\Lambda_{\omega}]\beta,\beta \rangle_{L^2}\\
 &\leq \langle [B_{\vep},\Lambda_{\omega}]^{-1}g_{\vep},g_{\vep} \rangle_{L^2} (||T^*\beta||^2+||S\beta||^2)
  \end{split}
\end{equation}
 from
$B_{\vep}\wedge \omega^q>0$, Cauchy--Schwarz inequality and Lemma \ref{basic estimate}.
 By the variant of Cauchy--Schwarz inequality
$$\left\langle\alpha_{1}+\alpha_{2}, \alpha_{1}+\alpha_{2}\right\rangle \leq\left\langle\alpha_{1}, \alpha_{1}\right\rangle+\left\langle\alpha_{2}, \alpha_{2}\right\rangle+c\left\langle\alpha_{1}, \alpha_{1}\right\rangle+\frac{1}{c}\left\langle\alpha_{2}, \alpha_{2}\right\rangle$$
we have
\begin{equation}\label{two part}
\langle [B_{\vep},\blambda_{\omega}]^{-1}g_{\vep},g_{\vep} \rangle_{L^2}\leq (1+c)\langle [B_{\vep},\blambda_{\omega}]^{-1}g_{\vep}^{(1)},g_{\vep}^{(1)} \rangle_{L^2} + (1+1/c)\langle [B_{\vep},\blambda_{\omega}]^{-1}g_{\vep}^{(2)},g_{\vep}^{(2)} \rangle_{L^2},
\end{equation}
where  $c$ is taken, for convenience,  to be the same $c$ as that in the definition of the  cutoff function $\theta$.
Then according to Lemmata \ref{inversion of order} and \ref{Twedge}, \eqref{B^1} and \eqref{g^1} give the estimate
\begin{equation}\label{first part}
 \begin{aligned}
  (1+c)\langle [B_{\vep},\blambda_{\omega}]^{-1}g_{\vep}^{(1)},g_{\vep}^{(1)} \rangle_{L^2} &\leq(1+c)\int_{\bomega_j}\frac{(|s|^{2}+\vep^2)^2}{4\vep^2\delta}\langle (T_{\dbar|s|}T_{\dbar|s|}^{*})^{-1}g_{\vep}^{(1)},g_{\vep}^{(1)}\rangle_{\omega} e^{-\psi}dV_{\omega}\\
  &= (1+c)\int_{\bomega_j}\frac{(|s|^{2}+\vep^2)^2}{4\vep^2\delta}\langle T_{\dbar|s|}^{-1}g_{\vep}^{(1)},T_{\dbar|s|}^{-1}g_{\vep}^{(1)}\rangle_{\omega} e^{-\psi}dV_{\omega}\\
  &=\frac{1+c}{\delta}\int_{\bomega_j}\frac{(\vep^{-2}|s|^{2}+1)^2\theta'(\vep^{-2}|s|^{2})^2}{\vep^2}|\widetilde{f}_{\infty}|^2_{\omega}e^{-\varphi}dV_{\omega}\\
  &\leq \frac{4(1+c)(1+2c)^2}{\delta}\int_{\bomega_j\cap\{\vep^{-2}|s|^{2}
  \leq 1\}}\vep^{-2}|\widetilde{f}_{\infty}|^2_{\omega}e^{-\varphi}dV_{\omega}.
  \end{aligned}
 \end{equation}

We denote the right-hand side  of the above inequality by $C_{c,\vep}$, whose limit is
\begin{equation}\label{limit}
 \frac{8\pi(1+c)(1+2c)^2}{\delta}\int_{\bomega_j\cap Y}\frac{|f|^{2}_{\omega}e^{-\varphi}}{|ds|^{2}_{\omega}e^{-\lambda}}dV_{Y,\omega} \quad \text{as}\quad \vep\to 0,
 \end{equation}
 by the Fubini theorem and Proposition \ref{tilde{f}}.\eqref{finfty=f}.

It's turn to estimate the term involving $g_{\vep}^{(2)}$. The relative compactness of $\bomega_j$ and the lower semi-continuity of $\sigma$ imply that $(q+1)\sigma'\sigma$ has a positive lower bound $\lambda_j$ which is independent of $\vep$. Then Lemma \ref{less than q eigen}, \eqref{B^2} and the boundedness of $\theta$ imply
 \[
 \begin{aligned}
&(1+\frac{1}{c})\left\langle [B_{\vep},\blambda_{\omega}]^{-1} g_{\varepsilon}^{(2)}, g_{\varepsilon}^{(2)}\right\rangle_{L^{2}}\\
\leq&(1+\frac{1}{c})\int_{\bomega_j}\frac{1}{(q+1)\sigma'\sigma} \langle s^{-1} \theta\left(\varepsilon^{-2}|s|^{2}\right) \bar{\partial} \widetilde{f}_{\infty},s^{-1} \theta\left(\varepsilon^{-2}|s|^{2}\right) \bar{\partial} \widetilde{f}_{\infty}\rangle_{\omega}
 e^{-\psi}dV_{\omega}\\
\leq& (1+\frac{1}{c})\frac{1}{\lambda_j}\int_{\bomega_j\cap\{\varepsilon^{-2}|s|^{2}\leq 1\}}|s^{-1}\dbar\widetilde{f}_{\infty}|_{\omega}^2e^{-\psi}dV_{\omega}.
\end{aligned}
\]
As $s^{-1}\dbar\widetilde{f}_{\infty}$ is of class $\mathcal{E}^k$ on $X$ and the volume of the integral region of above is $\sim O(\vep^2)$ due to
the relative compactness of $\bomega_j$,
\begin{equation}\label{second part}
(1+\frac{1}{c})\left\langle [B_{\vep},\blambda_{\omega}]^{-1} g_{\varepsilon}^{(2)}, g_{\varepsilon}^{(2)}\right\rangle_{L^{2}}\sim O(\vep^2),
\end{equation}
which depends on $c$ and $j$.

Denote $C_{c,  \varepsilon}+O(\vep^2)$ by $\frac{C_{\vep,j,c}}{\delta}$.
Then it follows
\begin{equation}
\left|\left\langle\beta, g_{\varepsilon}\right\rangle\right|_{L^{2}}^{2}
 \leq \frac{C_{\vep,j,c}}{\delta}\left(\left\|T^{*} \beta\right\|^{2}+\|S \beta\|^{2}\right)\label{boundedness}
\end{equation}
from \eqref{C-S estimate},\eqref{two part},\eqref{first part} and \eqref{second part}.

Set $\square:=T T^{*}+S^{*} S$. Then  we will solve the equation
\[
\square V_{\varepsilon}=g_{\varepsilon}
\]
in the standard way on the basis of the above estimate \eqref{boundedness}.

The Dirichlet semi-norm is defined as
\[
\|\beta\|_{\mathscr{H}}^{2}:=\left\|T^{*} \beta\right\|^{2}+\|S \beta\|^{2}
\]
for any smooth  $K_{X} \otimes L \otimes E^{*}$-valued $(0, q+1)$-form in the domain of $T^*$ and $S$. Since $\sigma$ is positive lower semi-continuous and $\sigma'>0$,  the original norm is dominated by the Dirichlet norm multiplied by some constant on the relatively compact domain $\Omega_j$ due to Lemmata \ref{basic estimate} and \ref{less than q eigen} and the strict positivity
\eqref{B^2} of $B_{\vep}$ on $X$. So the Dirichlet semi-norm is a norm on $\Omega_j$.

Let $\mathscr{H}$ denote the Hilbert space closure of the set of all smooth $K_X\otimes L \otimes E^{*}$-valued $(0, q+1)$-forms in the domain of $T^*$ and $S$. Consider the functional $\ell: \mathscr{H} \rightarrow \mathbb{C}$, defined by
\[
\ell(\beta):=\left(\beta, g_{\varepsilon}\right)=\int_{\Omega_j}\left\langle\beta, g_{\varepsilon}\right\rangle_{\omega} e^{-\varphi+\lambda} d V_{\omega}.
\]
Recall that the original $L^2$-norm is dominated by the Dirichlet norm for  smooth $K_X\otimes L \otimes E^{*}$-valued $(0, q+1)$-forms in the domain of $T^*$ and $S$. So it is easy to see that \eqref{boundedness} holds for every element of $\mathscr{H}$ by taking limits with respect to the graph norm by the density of $D_{T^*}\cap D_S\cap \mathcal{E}^{\infty}(\bar{\Omega})$ in $D_{T^*}\cap D_S$.  So $\ell$ is a  bounded linear functional on $\mathscr{H}$,  and  the $\mathscr{H}^{*}$-norm of $\ell$ is no more than $\delta^{-1} C_{\varepsilon,j,c}$. Then there exists $V_{\varepsilon,j} \in \mathscr{H}$  (here we omit the  subscript $c$ for $V_{\varepsilon,j}$ due to that the $c<<1$ is fixed until Remark \ref{tend 0}) such that
$$\left\|V_{\varepsilon,j}\right\|_{\mathscr{H}}^{2}=\|\ell\|_{\mathscr{H}^{*}}^{2} \leq \delta^{-1} C_{\varepsilon,j,c}\quad \text{and}\quad  \left(g_{\varepsilon}, \beta\right)=\left(T^{*} V_{\varepsilon,j}, T^{*} \beta\right)+\left(S V_{\varepsilon,j}, S \beta\right)$$
by the Riesz representation theorem.
 The latter defines the meaning of
 $$\square V_{\varepsilon,j}=g_{\varepsilon}$$
 in the weak sense, as $\beta$ goes through all the smooth compactly supported forms. Moreover, Lemma \ref{ellipticity} and $g_{\vep}\in \mathcal{E}^k$ tell us that
 %$\square V_{\varepsilon,j}=g_{\varepsilon}$ in the usual sense and that
 $V_{\varepsilon,j}$ is of \footnote{On Stein manifolds, McNeal--Varolin \cite[p. 437]{M-V19} can take the $g_{\vep}$ to be smooth and thereby get a smooth $V_{\vep,j}$.}  class $\mathcal{E}^{k_1+1}$ on $\Omega_j$ for some $k_1\geq 5$ (e.g.,  \cite[Chapter 3, Theorem 3.54]{Au82}).  As $S\circ T=0$ and $S g_{\varepsilon}=\sqrt{\tau} \bar{\partial} g_{\varepsilon}=0$,  we find that
\[
0=\left(S \square V_{\varepsilon,j}, S V_{\varepsilon,j}\right)=\left\|S^{*} S V_{\varepsilon,j}\right\|^{2}
\]
and thus
$\left\|S V_{\varepsilon,j}\right\|^{2}=\left(S^{*} S V_{\varepsilon,j}, V_{\varepsilon,j}\right)=0$.

Now in conclusion we obtain an  $L\otimes E^{*}$-valued $(n, q+1)$-form $V_{\varepsilon,,j }$ of class $\mathcal{E}^{k_1+1}$ such that
\[
\square V_{\varepsilon,j}=g_{\varepsilon} \quad \text { and } \quad\left\|V_{\varepsilon,j}\right\|_{\mathscr{H}}^{2} \leq \frac{C_{\varepsilon,j,c}}{\delta}.
\]
Set $v_{\varepsilon,j}:=T^{*} V_{\varepsilon,j}$. It follows that
\[
T v_{\varepsilon,j}=\square V_{\varepsilon,j}=g_{\varepsilon}.
\]
Then we have proved the following theorem.
 \begin{thm}\label{solution}
The equation $T v_{\vep, j}=g_{\varepsilon}$ has a  solution $v_{\varepsilon, j}\in\mathcal{E}^{k_1}(\bomega_j, L\otimes E^*\otimes\Lambda^{n,q}T^\star_X)$ satisfying the $L^{2}$-estimate
\[
\int_{\Omega_j}\left|v_{\varepsilon, j}\right|_{\omega}^{2} e^{-\varphi+\lambda} d V_{\omega} \leq \frac{C_{\varepsilon,j,c}}{\delta}.
\]
\end{thm}

\subsection{Construction of an $\mathcal{E}^{k_1}$ extension on $\Omega_{j}$ with  uniform $L^{2}$ bound}\label{sub-3.4}
Set
$$
u_{\varepsilon, j}:=\theta\left(\varepsilon^{-2}|s|^{2}\right) \widetilde{f}_{\infty}-\sqrt{\tau+A} v_{\varepsilon, j} \otimes s.
$$
Then
$$
u_{\varepsilon, j}\in\mathcal{E}^{k_1}(\bomega_j, L\otimes\Lambda^{n,q}T^\star_X),\quad  \left.u_{\varepsilon, j}\right|_{Y}=f \quad \text { and } \quad \bar{\partial} u_{\varepsilon, j}=s \otimes\left(g_{\varepsilon}-T v_{\varepsilon, j}\right)=0.
$$
Since $\theta\left(\varepsilon^{-2}|s|^{2}\right)$ is bounded and supported on a set whose measure tends to $0$ with $\varepsilon\to 0$ and $\Omega_j$ is relatively compact, there exists $\varepsilon_{j}>0$ sufficiently small so that whenever $\varepsilon \leq \varepsilon_{j},$ one has
$$
\begin{aligned}
\int_{\Omega_{j}}\left|u_{\varepsilon, j}\right|_{\omega}^{2} e^{-\varphi} d V_{\omega} &=(1+o(1)) \int_{\Omega_{j}}(\tau+A)\left|v_{\varepsilon, j}\right|_{\omega}^{2}\left|s\right|^{2} e^{-\varphi} d V_{\omega} \\
&=(1+o(1)) \int_{\Omega_{j}}\left(e^{v}(\tau+A)\right)\left|v_{\varepsilon, j}\right|_{\omega}^{2} e^{-\varphi+\lambda} d V_{\omega},
\end{aligned}
$$
where the infinitesimaali above is as $\varepsilon\to 0$.

Now
\begin{equation*}
\begin{aligned}
e^{v}(\tau+A)&=|s|^2(2e^{a-1}+2+\log(2e^{a-1}-1))\\
& \leq |s|^2\cdot 4e^{a-1}\leq4e^{\gamma-1},
\end{aligned}
\end{equation*}
where $0<\delta\leq1$.  It follows that for some sufficiently small $\varepsilon_{j}$, the estimate

\begin{equation}\label{universal C}
\begin{aligned}
\int_{\Omega_{j}}\left|u_{\varepsilon, j}\right|_{\omega}^{2} e^{-\varphi} d V_{\omega} & \leq(1+o(1)) 4e^{\gamma-1}\frac{C_{\vep,j,c}}{\delta} \\
& \leq \frac{C}{\delta} \int_{Y} \frac{|f|_{\omega}^{2} e^{-\varphi}}{\left|d s\right|^{2}_{\omega} e^{-\lambda}} dV_{Y,\omega}
\end{aligned}
\end{equation}
holds for some universal $C>0$, as soon as $0<\varepsilon \leq \varepsilon_{j}$, due to \eqref{limit} and  \eqref{second part}.
Thus for any such $\varepsilon>0, u_{\vep, j}$ gives the  extension with the desired $L^2$ estimate  in $\Omega_{j}$. Write
$$
u_{j}:=u_{\vep_{j}, j}.
$$
In conclusion, for each $j$ we have found an $L\otimes K_X$-valued $(0, q)$-form $u_{j}$ of class $\mathcal{E}^{k_1}$ on $\Omega_{j}$ such that
 \begin{equation}\label{u_j}
\bar\partial u_{j}=0,\left.\quad u_{j}\right|_{Y \cap \Omega_{j}}=f, \quad \text { and } \quad \int_{\Omega_{j}}\left|u_{j}\right|^{2}_{\omega} e^{-\varphi} d V_{\omega} \leq \frac{C}{\delta} \int_{Y} \frac{|f|_{\omega}^{2} e^{-\varphi}}{\left|d s\right|^{2}_{\omega} e^{-\lambda}} dV_{Y,\omega}.
\end{equation}
In particular, the right-hand side is independent of $j$.

\subsection{A kind of minimization problem for some continuous extensions}
First we give an overview of this subsection. From the above process we conclude that as $j \rightarrow \infty$, $\varepsilon$ is necessarily constrained to be smaller and smaller. In accordance with the previous practice in $L^2$ extension theory, we would like to take the limit as $\varepsilon \rightarrow 0$ of the extensions obtained in the last paragraph. The trouble is this limit cannot be directly taken due to the singularity of $\tau$ as $\vep\to 0$. That is to say, the twisted $\dbar$-operators $T$ and $S$ become singular as $\vep\to 0$, and thereby create  a loss of control on the constant $C$ in the statement Theorem \ref{ambient extension}.
Furthermore, as the weak limit as $\vep\to 0$ may not be a smooth  extension of $f$, we must look for better ambient extensions on $\Omega_j$ to ensure that the weak limit of the sequence of ambient extensions is  a smooth extension.

Here we adopt the method of McNeal--Varolin \cite[$\S$ 4.4]{M-V19} to reduce the problem involving the twisted operators to the untwisted operator, thereby eliminating the dependence on $\tau$.

To attack this problem, we define a subspace which is introduced in  \cite[$\S$ 4.4]{M-V19} and in which we perform our minimization procedure.

Let $\mathscr{V}_{q}^{2}\left(\Omega_{j}\right)$ denote the Hilbert space closure of the set
of all  smooth $K_X\otimes L\otimes E^*$-valued $\dbar$-closed $(0, q)$-forms $\beta$  on $\Omega_{j}$ satisfying
$$
\int_{\Omega_{j}}|\beta|_{\omega}^{2} e^{-\varphi+\lambda} d V_{\omega}<+\infty,
$$
and then  $\mathscr{V}_{q}^{2}\left(\Omega_{j}\right)$ consists precisely of all those forms in $L^{2}\left(\omega, e^{-\varphi+\lambda}\right)$ that are $\bar{\partial}$-closed in the weak sense. Let $\mathscr{B}_{q}^{2}\left(\Omega_{j}\right)$ denote the closed unit ball in $\mathscr{V}_{q}^{2}\left(\Omega_{j}\right)$, i.e.,
$$
\beta \in \mathscr{B}_{q}^{2}\left(\Omega_{j}\right) \Longleftrightarrow \beta \in \mathscr{V}_{q}^{2}\left(\Omega_{j}\right) \text { and } \int_{\Omega_{j}}|\beta|_{\omega}^{2} e^{-\varphi+\lambda} d V_{\omega} \leq 1.
$$
Let us define the affine ball
$$
\mathscr{B}_{j}:=u_{j}+s \mathscr{B}_{q}^{2}\left(\Omega_{j}\right):=\left\{u_{j}+s \beta; \beta \in \mathscr{B}_{q}^{2}\left(\Omega_{j}\right)\right\} \subset L^{2}\left(\omega, e^{-\varphi}\right).
$$

Now there are three problems to solve:
\begin{enumerate}[(i)]
    \item
Every continuous form in $\mathscr{B}_{j}$ is an ambient extension of $f$.
    \item\label{question 2}
There exists a minimizer in $\mathscr{B}_{j}$.
\item\label{question 3}
The minimizer is smooth.
\end{enumerate}

Note that we can of course adopt the method of McNeal--Varolin \cite[(12),(13)]{M-V19} to  grasp the data along $Y$ of any continuous bundle-valued form by taking wedge with the current of integration $[Y]$  due to $\operatorname{Supp}[Y]=Y$.  Here we use a simpler direct method but not the current method to verify the ambient extension relationship.
\begin{prop}\label{is ambient extension}
Every continuous
form in $\mathscr{B}_{j}$ is an ambient extension of $f$.
\end{prop}
\begin{proof}
For any continuous form $g\in\mathscr{B}_{j}$, there exists a continuous  form $g_1$
(in the classical usual sense but not as an element in some $L^2$ space) such that $g=g_1$ a.e. on $\Omega_j$. We will prove that $g_1$ is an ambient extension of $f$. Let $g_1=u_{j}+s\beta\in\mathscr{B}_{j}$, then $\beta$ is a continuous
(in the classical usual sense) form divided by $s$.  Then we just need to prove
%Note that verifying the extension or restriction is essentially a local problem. That is to say,
that $(u_{j}+s\beta)(x)=f(x)$ for any $x\in\Omega_j\cap Y$.  For any fixed  $x\in\Omega_j\cap Y$, take any  local coordinate  $(U, z_1,z_2,\ldots,z_n)$ of $X$ around $x$.  Fix  holomorphic local frames $\sigma$ of $K_X\otimes L$
and $\theta$ of $E$ over $U$, respectively, such that $s=z_1\otimes\theta$.
Note that
$$Y\cap U=\{z_1=0\}, \quad \beta|_U=\sum_{|K|=q}\lambda_{K}d\bar{z}_K\otimes \sigma\otimes\theta^*,$$
where the multi-index $K$ is increasing.
Then it suffices to prove
$$\lim\limits _{z_1 \rightarrow 0} z_1 \lambda_{K}(z_1,z_2,\ldots,z_n)=0 \quad\text{for any increasing multi-index $K$},$$
since $u_j$ is the ambient extension of $f$.

According to the definition of $\mathscr{B}_{q}^{2}$, we know that $\beta$ and thereby $\lambda_{K}(z_1,z_2,\ldots,z_n)$ is $L^2$ integrable (possibly after shrinking the domain).
Then, by the integrability part of Fubini theorem (e.g.,  \cite[8.8.(c)  Theorem]{R87}), $|\lambda_{K}(z_1,z_2,\ldots,z_n)|^2$ is $L^1$ integrable with respect to $z_1$ for $(z_2,\ldots,z_n)$ a.e..  Due to Lemma
\ref{integrability}, there exists a sequence $\{z_{1,\nu}\}$ such that $\lambda_{K}(z_{1,\nu},z_2,\ldots,z_n)\sim o(|z_{1,\nu}|^{-1})$ for any fixed $(z_2,\ldots,z_n)$ a.e.. Then
$$\lim\limits _{z_1 \rightarrow 0} z_1 \lambda_{K}(z_1,z_2,\ldots,z_n)= 0$$
for $(z_2,\ldots,z_n)$  a.e. due to the continuity of  $z_1 \lambda_{K}(z_1,z_2,\ldots,z_n)$. The continuity of  $z_1 \lambda_{K}(z_1,z_2,\ldots,z_n)$
 again  implies that
 $$\lim\limits _{z_1 \rightarrow 0} z_1 \lambda_{K}(z_1,z_2,\ldots,z_n)= 0$$
for any $(z_2,\ldots,z_n)$.
\end{proof}

For the sake of completeness, we will give the following two propositions, which can be found in \cite[$\S$ 4.4]{M-V19}. They solve the second and
third problems  listed above, i.e., the existence \eqref{question 2} and the regularity \eqref{question 3} of the minimizer.
\begin{prop}\label{minimal ortho}
There exists an element of minimal norm $U_{j} \in \mathscr{B}_{j}$
and  $U_{j}$ is orthogonal to $s \mathscr{B}_{q}^{2}\left(\Omega_{j}\right)$.
\end{prop}
\begin{proof}
By  Fatou's lemma and Lemma \ref{extend across}, $\mathscr{B}_{j}$ is a  closed subset of the Hilbert space $L^{2}(\omega, e^{-\varphi})$. Furthermore, it is apparently convex. So  $\mathscr{B}_{j}$ has an element of minimal norm $U_{j}$.

Suppose that there exists $\beta_{0} \in s \mathscr{B}_{q}^{2}\left(\Omega_{j}\right)$ such that $\left(U_{j}, \beta_{0}\right)=c \neq 0$.
Consider the form $$\alpha:=\frac{c \beta_{0}}{\left\|\beta_{0}\right\|^{2}} \in s \mathscr{V}_{q}^{2}$$ and set $\widetilde{U}_{j}=U_{j}-\alpha$.  Then $\widetilde{U}_{j} \in \mathscr{B}_{j},$ but
$\left\|\widetilde{U}_{j}\right\|^{2}=\left\|U_{j}\right\|^{2}-\frac{|c|^{2}}{\left\|\beta_{0}\right\|^{2}}$.  This contradicts the minimality of $\left\|U_{j}\right\|$.
\end{proof}

\begin{prop}\label{dbarstar=0}
$\bar{\partial}^{*} U_{j}=0$ in the sense of currents.
\end{prop}

\begin{proof} For any $\alpha\in \mathcal{E}_0^{\infty}(\bomega_j-Y, L^*\otimes\Lambda^{n,q-1}T^\star_X)$, $s^{-1}\dbar\alpha\in \mathscr{B}_{q}^{2}\left(\Omega_{j}\right)$ possibly after shrinking $\alpha$ by some constant due to the smoothness of $s^{-1}$ on $\bomega_j-Y$.  By Proposition \ref{minimal ortho},
$$(U_j,\dbar\alpha)=(U_j,ss^{-1}\dbar\alpha)=0.$$
So $\bar{\partial}^{*} U_{j}=0$,  in the sense of currents, on $\Omega_{j}-Y$.  The minimality of $U_j$ implies easily $U_j\in L_{loc}^2$. An adaptation of the proof of Lemma \ref{extend across} (=\cite[Lemme 6.9]{Dm82}) to the $\bar{\partial}^{*}$-equation  or a direct application of \cite[Proposition 4.8]{B18} yields that the above equation can extend across $Y$, i.e.,  $\bar{\partial}^{*} U_{j}=0$ on $\Omega_{j}$.
\end{proof}

It follows from  Proposition \ref{dbarstar=0} that $$\square_{0} U_{j}=0$$ in the sense of currents, where
 $\square_{0}=\dbar\dbar^*+\dbar^*\dbar$ denotes the untwisted $\dbar$-Laplacian unrelated with $\tau$. So it follows that $U_{j}$ is smooth on $\Omega_{j}$ from the ellipticity of the Laplacian $\square_{0}$ due to the smoothness of the metric $e^{-\varphi}$.  Thus, Proposition \ref{is ambient extension} gives that $U_{j}$ is an extension of $f$.  Moreover, by the estimate \eqref{u_j} for $u_{j}$ and the minimality of $U_{j}$, we have
\begin{equation}\label{U_j}
\int_{\Omega_{j}}\left|U_{j}\right|^{2}_{\omega} e^{-\varphi} d V_{\omega} \leq \frac{C}{\delta} \int_{Y} \frac{|f|_{\omega}^{2} e^{-\varphi}}{\left|d s\right|^{2}_{\omega} e^{-\lambda}} dV_{Y,\omega}.
\end{equation}

\subsection{The end of the proof of Theorem \ref{ambient extension}}\label{sub-3.6}
Now we construct the desired extension. Since there is no something like the Montel property for holomorphic objects now, we cannot derive any pointwise convergence information. Then we  use the same method as \cite[$\S$ 4.5]{M-V19} since we have obtained the extension sequence $\{U_j\}$ naturally satisfying the current equations \cite[(16)(17)]{M-V19} of the ambient extension. For the sake of completeness, we show the specific  procedure (see \cite[p. 439, $\S$ 4.5]{M-V19} for more explanations).

 As $U_{j}$ is an ambient extension of $f$, it satisfies  the distribution equations
\begin{equation}\label{current1new}
U_j \wedge \frac{\sqrt{-1}}{2\pi} \partial \bar{\partial} \log \left|s\right|^{2}=f \wedge \frac{\sqrt{-1}}{2\pi} \partial \bar{\partial} \log \left|s\right|^{2}
\end{equation}
and
\begin{equation}\label{current2new}
(\frac{\partial}{\partial\bar{s} }\lrcorner U_j) \wedge \frac{\sqrt{-1}}{2\pi} \partial \bar{\partial} \log \left|s\right|^{2}=(\frac{\partial}{\partial\bar{s} }\lrcorner f) \wedge \frac{\sqrt{-1}}{2\pi} \partial \bar{\partial} \log \left|s\right|^{2},
\end{equation}
which have been derived in \cite[(12)(13)]{M-V19} due
essentially to the idea that we can grasp the data along $Y$ of any continuous bundle-valued form by taking wedge with the current of integration $[Y]$  since $\operatorname{Supp}[Y]=Y$.
Note that $s$ here is considered as the holomorphic function  coefficient with respect to some corresponding local frame of $E$. Of course, $f$ is only defined on $Y$. However, we can extend it smoothly in an arbitrary way to the object of the same type on $\Omega_j$. Then the support of $[Y]$ forces the above current equation  to be  well defined on $\Omega_j$ which does not depend on the choice of the  extension of $f$.

According to \eqref{U_j}, Alaoglu's Theorem shows that (possibly a subsequence of) $\left\{U_{j}\right\}$ converges weakly to some $U$ on $X$ with the $L^{2}$ estimate
\begin{equation}\label{U}
\int_{X}|U|^{2}_{\omega} e^{-\varphi} d V_{\omega}\leq \lim\limits \inf ||U_j||_{L^2} \leq \frac{C}{\delta} \int_{Y} \frac{|f|_{\omega}^{2} e^{-\varphi}}{\left|d s\right|^{2}_{\omega} e^{-\lambda}} dV_{Y,\omega}.
\end{equation}
Then $\square_{0} U=0$ due to that $\square_{0} U_j=0$,  and thus $U$ is smooth.

Now we claim that $U$ is our desired extension, i.e.,  $U$ satisfies the distribution equations \eqref{current1new} and \eqref{current2new}.
Note that
$$
U_j \wedge \frac{\sqrt{-1}}{2\pi} \partial \bar{\partial} \log \left|s\right|^{2}=-\dbar(U_j \wedge \frac{\sqrt{-1}}{2\pi} \frac{ds}{s})
$$
and
$$
(\frac{\partial}{\partial\bar{s} }\lrcorner U_j) \wedge \frac{\sqrt{-1}}{2\pi} \partial \bar{\partial} \log \left|s\right|^{2}=
\pm\left(\dbar(\frac{\partial}{\partial\bar{s}}\lrcorner U_j)\wedge\frac{ds}{s}-\dbar((\frac{\partial}{\partial\bar{s}}\lrcorner U_j)\wedge\frac{ds}{s})\right).
$$
Then  it suffices to show that
\begin{equation}\label{locally integrable}
U_j,\quad    \dbar(\frac{\partial}{\partial\bar{s}}\lrcorner U_j)   \quad \text{and}\quad   \frac{\partial}{\partial\bar{s}}\lrcorner U_j
\end{equation}
%\begin{equation}\label{locally integrable}
%U_j \wedge \frac{\sqrt{-1}}{2\pi} \frac{ds}{s},\quad   (\frac{\partial}{\partial\bar{s}}\lrcorner U_j)\wedge\frac{ds}{s}\quad \text{and}\quad \dbar(\frac{\partial}{\partial\bar{s}}\lrcorner U_j)\wedge\frac{ds}{s}
%\end{equation}
are locally  uniformly bounded in $j$, for  showing that $U$ satisfies equations \eqref{current1new} and \eqref{current2new}, since $\frac{ds}{s}$ is locally integrable. In fact, take \eqref{current1new} for example, in the spirit of the argument (e.g.,  \cite[p. 383]{Dm12}) of Lemma \ref{extend across}, noting that $s^{-1}$ multiplied by any smooth function which is  compactly supported outside $Y$ is  smoothly compactly supported on $X$, we can   imply the convergence of $\{U_j \wedge \frac{\sqrt{-1}}{2\pi} \frac{ds}{s}\}$ to the current $U \wedge \frac{\sqrt{-1}}{2\pi} \frac{ds}{s}$ due to the  locally uniformly boundedness of  $U_j$.  Then the weak continuity of $\dbar$ helps us to get  \eqref{current1new}.

%G\aa rding's inequality
 The ellipticity of $\square_{0}$,  basic elliptic estimate (e.g.,  \cite[Theorem A.3.2]{MM07}),  Sobolev embedding theorem (e.g.,  \cite[Theorem A.3.1-(b)]{MM07}) and  \eqref{U_j}  together tell us that for every compact set $K$, there exists some $j_{0}=j_{0}(K)>0$ and $C_{K}>0$ such that
$$
\left\|\left|U_{j}\right| e^{-\varphi / 2}\right\|_{\mathscr{C}^1(K)} \leq C_{K}
$$
for all $j \geq j_{0} .$ The index $j_{0}$ is large enough to make sure that $U_{j}$ is defined
on $\Omega_{j}$.
%It follows that for any compact subset $K \subset \subset X$ and $j \geq j_{0}$,
%$$
%\int_{K}\left|U_{j} \wedge \frac{d s}{s}\right| e^{-\varphi / 2} d V_{\omega} \leq C_{K} \int_{K}\left|\frac{d s}{s}\right| d V_{\omega},
%$$
%and thus the $L^{1}$-norm of $U_{j} \wedge \frac{d s}{s}$ on $K$ is bounded by a constant independent
%of $j$.
The similar argument shows that the other terms in \eqref{locally integrable}  are locally uniformly bounded.  Now we can take limits in the distribution equations \eqref{current1new} and \eqref{current2new} to conclude that $U$ also satisfies these equations. Thus $U$ is our desired extension with the estimate \eqref{U}.

\rem\label{tend 0}Note that we can make a better control on our extension $U$. In fact, from \eqref{boundedness}, Theorem \ref{solution} and \eqref{universal C},  we know that the universal constant $C$ in \eqref{u_j} and \eqref{U} depends on $c$. If we take a smaller $c$ in \eqref{two part}, we can get a smaller $C_{\vep,j,c}$ in \eqref{boundedness} according to \eqref{first part},\eqref{limit} and \eqref{second part}. Then we can get a smaller universal constant $C$ in \eqref{universal C} and thereby a better control in \eqref{U} for our extension $U$. However, the constant in \eqref{U} is not a sharp one and then taking a smaller $c$ may not give more useful information.

\appendix
\section{Alternative proof of Theorem \ref{intrinsic extension}}\label{app}
Theorem \ref{ambient extension} can easily imply Theorem \ref{intrinsic extension} as stated in Section \ref{intro}. Now we  present  a sketch of a direct proof of Theorem \ref{intrinsic extension}.

Just as in Subsection \ref{sect0.1}, we can glue local extensions from
 Theorem \ref{mv Intrinsic} and use the same method of raising the regularity of $s^{-1}\dbar \widetilde{f}_{\infty}$, to obtain:
\begin{prop}[{\cite[Lemma 3.1]{Koz11}}]\label{tilde{f}-2}
For any $k\geq 0$, there exists a smooth section
$$\widetilde f_\infty\in \mathcal E^\infty(X,\Lambda^{n,q}T_X^\star\otimes L)$$
such that
\begin{enumerate}[{\rm (}a{\rm )}]
\item $\widetilde f_\infty$ is the intrinsic extension of $u$,
\item $|\widetilde f_\infty|_{\omega,L}=|u|_{\omega,L}$ at every point of $Y$,
\item $\bar\partial\widetilde f_\infty=0$ at every point of $Y$,
\item $s^{-1} \bar\partial\widetilde f_\infty\in \mathcal {E}^k(X,\Lambda^{n,q+1}T^\star_ X \otimes L\otimes \mathcal{O}_{X}(-Y))$.
\end{enumerate}
\end{prop}
From now on we fix $k\geq n+6$, where $n$ is the complex dimension of $X$.

Using exactly the same constructions and calculations as in Subsections  \ref{sub-3.2}-\ref{sub-3.4}, for each $j$, we
can obtain an $L\otimes K_X$-valued $(0, q)$-form $u_{j}$ of class $\mathcal{E}^{k_1}$ on $\Omega_{j}$ such that
$$\label{u_j-2}
\bar\partial u_{j}=0,\quad \iota^*u_{j}=u, \quad \text { and } \quad \int_{\Omega_{j}}\left|u_{j}\right|^{2}_{\omega} e^{-\varphi} d V_{\omega} \leq \frac{C}{\delta} \int_{Y} \frac{|u|_{\omega}^{2} e^{-\varphi}}{\left|d s\right|^{2}_{\omega} e^{-\lambda}} dV_{Y,\omega}
$$
with $k_1\geq 5$.
In particular, the right-hand side is independent of $j$.

Let $\mathscr{V}_{q}^{2}\left(\Omega_{j}\right)$ denote the Hilbert space closure of the set
of all  smooth $K_X\otimes L\otimes E^*$-valued $\dbar$-closed $(0, q)$-forms $\beta$ on $\Omega_{j}$ satisfying
$$
\int_{\Omega_{j}}|\beta|_{\omega}^{2} e^{-\varphi+\lambda} d V_{\omega}<+\infty,
$$
and then  $\mathscr{V}_{q}^{2}\left(\Omega_{j}\right)$ consists precisely of all those forms in $L^{2}\left(\omega, e^{-\varphi+\lambda}\right)$ that are $\bar{\partial}$-closed in the weak sense. Let $\mathscr{B}_{q}^{2}\left(\Omega_{j}\right)$ denote the closed unit ball in $\mathscr{V}_{q}^{2}\left(\Omega_{j}\right)$, i.e.,
$$
\beta \in \mathscr{B}_{q}^{2}\left(\Omega_{j}\right) \Longleftrightarrow \beta \in \mathscr{V}_{q}^{2}\left(\Omega_{j}\right) \text { and } \int_{\Omega_{j}}|\beta|_{\omega}^{2} e^{-\varphi+\lambda} dV_{\omega} \leq 1.
$$
Let us define the affine ball
$$
\mathscr{B}_{j}:=u_{j}+s \mathscr{B}_{q}^{2}\left(\Omega_{j}\right):=\left\{u_{j}+s \beta; \beta \in \mathscr{B}_{q}^{2}\left(\Omega_{j}\right)\right\} \subset L^{2}\left(\omega, e^{-\varphi}\right).
$$

Then we have  the  following proposition.
\begin{prop}\label{is intrinsic extension}
Every continuous form in $\mathscr{B}_{j}$ is an intrinsic extension of $u$.
\end{prop}
\begin{proof} The proof is quite similar with that of Proposition \ref{is ambient extension}. It suffices to prove that $\iota^*(u_{j}+s\beta)(x)=u(x)$ for any $x\in\Omega_j\cap Y$ and any continuous $u_{j}+s\beta$ in $\mathscr{B}_{j}$. For any fixed  $x\in\Omega_j\cap Y$, take any  local coordinate  $(U, z_1,z_2,\ldots,z_n)$ of $X$ around $x$. Fix  holomorphic local frames $\sigma$ of $K_X\otimes L$
and $\theta$ of $E$ over $U$, respectively, such that $s=z_1\otimes\theta$. Note that
$$Y\cap U=\{z_1=0\}, \quad \beta|_U=\sum_{|J|=q-1,1\notin J}\lambda^{(1)}_{J}d\bar{z}_1\wedge d\bar{z}_J\otimes \sigma\otimes\theta^*+\sum_{|K|=q,1\notin K}\lambda^{(2)}_{K}d\bar{z}_K\otimes \sigma\otimes\theta^*,$$
where the multi-indices $J$ and $K$ are increasing.

Then
$$\iota^*\beta=\sum_{|K|=q,1\notin K}\lambda^{(2)}_{K}(0,z_2,\ldots,z_n)d\bar{z}_K\otimes \sigma\circ\iota\otimes\theta^*\circ\iota.$$
So it suffices to prove
$$\lim\limits _{z_1 \rightarrow 0} z_1 \lambda^{(2)}_{K}(z_1,z_2,\ldots,z_n)= 0 \quad\text{for any multi-index $K$},$$
since $u_j$ is the intrinsic extension of $u$.
The leftover argument proceeds just with $\lambda_{K}(z_1,z_2,\ldots,z_n)$ in the proof of  of Proposition \ref{is ambient extension} replaced by $\lambda^{(2)}_{K}(z_1,z_2,\ldots,z_n)$ here.
\end{proof}

\rem The current equation method in the proof of  \cite[Proposition 4.4]{M-V19} can also work to prove Proposition \ref{is intrinsic extension}. In fact, the local expressions of $\beta$ and $\iota^*\beta$ in the above proof can be used to deduce a current equation  characterizing
the intrinsic restriction due to that $\operatorname{Supp}[Y]=Y$. That is,
 a continuous $K_X\otimes L$-valued $(0, q)$-form $U$  on $X$ is the intrinsic extension of a smooth  $K_X\otimes L|_Y$-valued  $(0,q)$-form $u$ on $Y$ if and only if locally
\begin{equation}\label{intrinsic current}
U \wedge \sqrt{-1} \partial \bar{\partial} \log \left|s\right|^{2}=u \wedge \sqrt{-1} \partial \bar{\partial} \log \left|s\right|^{2}
\end{equation}
as $K_X\otimes L$-valued $(1, q+1)$-currents of order $0$. Here we adopt the definition of order of currents in \cite[$\S$\ 2 of Chapter 1]{Dm12}.  Note that $|s|$ here is considered as the holomorphic function with respect to the corresponding local trivialization. Despite $u$ is only defined on $Y$, the above current equation  is   well defined on $X$ due to the same reason as \eqref{current1new} and \eqref{current2new}.  The proof of \eqref{intrinsic current} is similar to  the proof of \cite[(12),(13)]{M-V19}  characterizing the ambient restriction. Then we can  also use this characterization \eqref{intrinsic current} to verify  Proposition \ref{is intrinsic extension} as the  proof of \cite[Proposition 4.4]{M-V19}.

Using Proposition \ref{is intrinsic extension} and the similar results to  Propositions \ref{minimal ortho} and \ref{dbarstar=0},
we can obtain an element $U_j$ of minimal norm in $\mathscr{B}_{j}$  satisfying $\square_{0} U_{j}=0$. Then $U_j$ is smooth and thereby the intrinsic extension of $u$ on $\Omega_j$ by Proposition \ref{is intrinsic extension}. Of course,  $U_j$ satisfies \eqref{intrinsic current}. At last, by the same method of extracting weak limits of some subsequence of  ${U_j}$ as $j\to\infty$ as in Subsection \ref{sub-3.6}, we can obtain our desired extension.

\section*{Acknowledgement}
Both authors would like to thank Professor  Dror Varolin,  for explaining us the details on their article \cite{M-V19}, and pointing out an important mistake in our draft, and giving many useful suggestions, especially on Examples \ref{ex1} and \ref{ex2}. Furthermore, we also thank Professors  Bo Berndtsson, Vincent Koziarz  for explaining us details on their articles \cite{B12,Koz11}, respectively, and Professor Langfeng Zhu for his help in learning  $L^2$  theory, and Zheng Yuan for pointing out an inaccuracy,  and also Professors Chen-Yu Chi and J.-P. Demailly for their two discussions during the second author's visit to Institut Fourier, Universit\'{e} Grenoble Alpes in April 2019 on this topic of extension theorem.  Moreover, the first author would like to thank Runze Zhang for many  discussions on $L^2$ extension theory, as well
as Yongpan Zou and Houwang Liu  for many discussions  on a seminar about $L^2$ methods.

\end{document}